\documentclass[a4paper,11pt]{amsart}

\usepackage{graphicx, amssymb}  
\usepackage{mathrsfs}   
\usepackage[numbers, sort&compress]{natbib}
\usepackage{url,hyperref} 
\usepackage[inner=2.5cm,outer=2.5cm,top=2.5cm,bottom=2.5cm,includehead,includefoot]{geometry}
\usepackage{float}
\hypersetup{citecolor=red, linkcolor=blue, colorlinks=true}

\newtheorem{thm}{Theorem}[section]
\newtheorem{prop} [thm]{Proposition}
\newtheorem{cor} [thm]{Corollary}
 \newtheorem{lemma} [thm]{Lemma}

\renewcommand\leq{\leqslant} 
\renewcommand\geq{\geqslant}

\renewcommand\mod{\bmod} 

\newcommand{\wt}[1]{|#1|}
\DeclareMathOperator{\Aut}{Aut}
\DeclareMathOperator{\Sym}{Sym}
\DeclareMathOperator{\PGaL}{P\Gamma L}
\DeclareMathOperator{\PSL}{PSL}
\DeclareMathOperator{\PSigmaL}{P\Sigma L}
\DeclareMathOperator{\POmega}{P\Omega}
\DeclareMathOperator{\Sp}{Sp}

\title[Locally triangular graphs and rectagraphs with symmetry]{Locally triangular graphs\\ and rectagraphs with symmetry}
\author[J. Bamberg, A. Devillers,  J. B. Fawcett, and C. E. Praeger]{John Bamberg, Alice Devillers,  Joanna B. Fawcett, \\ and  Cheryl E. Praeger}

\address{
Centre for the Mathematics of Symmetry and Computation\\
School of Mathematics and Statistics\\
The University of Western Australia\\
35 Stirling Highway, Crawley, W.A. 6009, Australia.\newline
Email: \texttt{\{john.bamberg, alice.devillers, joanna.fawcett, cheryl.praeger$^\dag$\}@uwa.edu.au}\\
\newline $^\dag$ Also affiliated with King Abdulaziz University, Jeddah, Saudi Arabia.
}


\begin{document}

\begin{abstract}
Locally triangular graphs are known to be halved graphs of bipartite rectagraphs, which are
connected triangle-free graphs in which every $2$-arc lies in a unique quadrangle. A graph $\Gamma$
is locally rank 3 if there exists $G\leq \Aut(\Gamma)$ such that for each vertex $u$, the
permutation group induced by the vertex stabiliser $G_u$ on the neighbourhood $\Gamma(u)$ is
transitive of rank 3. One natural place to seek locally rank 3 graphs is among the locally
triangular graphs, where every induced neighbourhood graph is isomorphic to a triangular graph
$T_n$. This is because the graph $T_n$, which has vertex set the $2$-subsets of $\{1,\ldots,n\}$ and
edge set the pairs of $2$-subsets intersecting at one point, admits a rank 3 group of automorphisms.
In this paper, we classify the locally $4$-homogeneous rectagraphs under some additional structural
assumptions. We then use this result to classify the connected locally triangular graphs that are
also locally rank 3.
\end{abstract}

\maketitle

%%%%%%%%%%%%%%%%%%%%%%%%%%%%%%%%%%%%%%%%%%%%%%%%%%%
%
%	1. Introduction
%
%%%%%%%%%%%%%%%%%%%%%%%%%%%%%%%%%%%%%%%%%%%%%%%%%%%

\section{Introduction}
\label{intro}

A finite simple undirected graph $\Gamma$ is \textit{locally} $\Delta$ (or $\mathscr{C}$) for some
graph $\Delta$ (or class of graphs $\mathscr{C}$) if for every vertex $u\in V\Gamma$, the graph
induced by the neighbourhood $\Gamma(u)$ is isomorphic to $\Delta$ (or some graph in $\mathscr{C}$).
There is a well-established tradition of classifying such graphs. Extending this concept, there is a
trend of studying graphs $\Gamma$ for which some set of vertices related to $u$ is highly symmetric
for each $u\in V\Gamma$, such as locally projective graphs \cite{Iva1999}, or locally $s$-arc
transitive graphs \cite{GiuLiPra2004,Lee2009,van2011}.

 A graph $\Gamma$ is \textit{locally $2$-arc transitive} if $\Gamma$ contains a $2$-arc and there
 exists $G\leq \Aut(\Gamma)$ such that, for every $u\in V\Gamma$, the stabiliser $G_u$ acts
 transitively on the $2$-arcs starting at $u$, where a \textit{$2$-arc} is a tuple of vertices
 $(u,v,w)$ such that $u\neq w$ and $u,w\in\Gamma(v)$. A $2$-arc $(u,v,w)$ is either a
 \textit{triangle} when $u$ and $w$ are adjacent, or a \textit{$2$-geodesic} when $u$ and $w$ are at
 distance 2 apart. Connected non-complete graphs with girth 3 contain triangles and $2$-geodesics
 with the same initial vertex and are therefore never locally $2$-arc transitive, but for the most
 symmetrical of these graphs, the $2$-arcs with any given initial vertex $u$ fall into two
 $\Aut(\Gamma)_u$-orbits, namely the triangles and the $2$-geodesics. This is equivalent to the
 permutation group induced by $\Aut(\Gamma)_u$ on ${\Gamma(u)}$ being transitive of rank 3 for all
 $u\in V\Gamma$ (cf. Proposition \ref{2 orbits}).

Motivated by this observation, we introduce the following definition: a graph $\Gamma$ is
\textit{locally rank 3 with respect to $G$} if $\Gamma$ has no vertices with valency $0$ and $G\leq
\Aut(\Gamma)$ such that, for all $u\in V\Gamma$, the permutation group induced by $G_u$ on
$\Gamma(u)$ is transitive of rank 3 (cf. \S \ref{loc rank 3} for the definition of rank). We also
say that $\Gamma$ is \textit{locally rank 3} if it is locally rank 3 with respect to some $G$. Note that there exist connected graphs that are both locally rank 3 and locally
$2$-arc transitive (cf. \S \ref{loc rank 3}); necessarily,  these graphs are either complete or have girth at least 4.

In this paper, we study and classify a family of locally rank 3 graphs associated with the class of
rectagraphs (defined below). Note that locally disconnected locally rank 3 graphs were recently
analysed in a more general setting by Devillers et al. \cite{DevJinPra2013}, but the graphs we are
interested in are locally connected.

%%%%%%%%%%%%%%%%%%%%%%%%%%%%%%%%%%%%%%%%%%%%%%%%%%%
%
%	Locally triangular and locally rank 3 graphs
%
%%%%%%%%%%%%%%%%%%%%%%%%%%%%%%%%%%%%%%%%%%%%%%%%%%%

\subsection*{Locally triangular and locally rank 3 graphs} The most basic example of a rank 3
permutation group is the action of the symmetric group $S_n$ on the set $\tbinom{n}{2}$ of
$2$-subsets of $\{1,\ldots,n\}$. There are two graphs corresponding to this action. One is the
\textit{triangular graph} $T_n$, where a pair of $2$-subsets are adjacent whenever they intersect at
exactly one point, and the other is the complement $\overline{T}_n$ of $T_n$.

Remarkably, connected locally $\overline{T}_n$ graphs were completely classified by Hall and Shult
\cite{HalShu1985}. These graphs are very well behaved, for $\overline{T}_{n+2}$ is itself locally
$\overline{T}_n$, and for $n\geq 7$, it is the only such connected graph. In particular, locally
$\overline{T}_n$ graphs have bounded diameter, and we will see that they are all locally rank 3 (cf.
Corollary \ref{Hall Shult}). Note that $\overline{T}_5$ is the Petersen graph, so this
classification includes locally Petersen graphs, first classified by Hall \cite{Hall1980}.

However, the behaviour of locally $T_n$ graphs is much wilder. For example, the graph $T_n$ is not
locally $T_m$ for any $m$, nor is it ever locally rank 3. Moreover, the diameter of locally
triangular graphs is unbounded, for the halved $n$-cube is locally $T_n$ and has diameter $\lfloor
n/2\rfloor$. In fact, every connected component of the distance 2 graph of a coset graph of a linear
code over $\mathbb{F}_2^n$ with minimum distance at least seven is locally $T_n$ (cf. Lemma
\ref{code tri}), but this graph is rarely locally rank 3 (cf. Theorem \ref{rank 3} and Proposition
\ref{G}).

We say that a graph is \textit{locally triangular} if it is locally $\mathscr{C}$, where
$\mathscr{C}$ is the class of all triangular graphs. It turns out that a connected locally
triangular graph is always locally $T_n$ for some $n$ \cite[Proposition 4.3.9]{BroCohNeu1989}.
Strongly regular locally $T_n$ graphs were classified in \cite{Mak2001}, and 1-homogeneous locally
triangular graphs were classified in \cite[Theorem 4.4]{JurKoo2003}. The graphs in these
classifications all appear in our first main result Theorem \ref{main}, where we completely classify
the connected graphs that are locally triangular and locally rank 3. 

\begin{thm}
\label{main} A connected graph $\Gamma$ is locally rank $3$ and locally triangular if and only if
$\Gamma$ is the halved graph of one of the following bipartite graphs.
\begin{itemize}
\item[(i)] The  $n$-cube $Q_n$ where $n\geq 3$.
\item[(ii)] The  folded  $n$-cube $\Box_n$ where  $n$ is even and $n\geq 8$.
\item[(iii)]The bipartite  double of the  coset graph of the   binary Golay code $C_{23}$.
\item[(iv)]  The  coset graph  of the extended binary Golay code $C_{24}$.
\end{itemize}
Moreover, $\Gamma$ is locally rank $3$ with respect to $G\leq \Aut(\Gamma)$ if and only if $G$ is
listed in Table \ref{tab: group}.
\end{thm}

\begin{table}[!h]
\centering
\begin{tabular}{ l l l  }
\hline
$\Gamma$ & $n$ & $G$ \\
\hline 
$ \tfrac{1}{2}Q_n$ & $n\geq 5$ & $2^{n-1}\rtimes S_n$, $2^{n-1}\rtimes A_n$ \\
& $3$ & $A_4$ \\
& $4$ & $2^4\rtimes S_4$, $2^3\rtimes S_4$, $(2^3\rtimes A_4).2$ \\
& $9$ & $2^8\rtimes \PGaL_2(8)$ \\
& $11,12,23,24$ & $2^{n-1}\rtimes M_n$ \\ 
$\tfrac{1}{2}\Box_n$ & $n\geq 8$ even & $2^{n-2}\rtimes S_n$, $2^{n-2}\rtimes A_n$ \\
& $12,24$ & $2^{n-2}\rtimes M_n$ \\ 
$\tfrac{1}{2}\Gamma(C_{23}).2$ & 23 & $2^{11}\rtimes M_{23}$ \\
$ \tfrac{1}{2}\Gamma(C_{24})$ & 24 & $2^{11}\rtimes M_{24}$ \\
\hline
\end{tabular}
\caption{$G\leq \Aut(\Gamma)$ for which $\Gamma$ is locally rank 3}
\label{tab: group}
\end{table}

The graphs of Theorem \ref{main} are described in \S \ref{Thm 1,2} (see also \S \ref{notation}). They are all distance-transitive
graphs of valency $\tbinom{n}{2}$ where $n=23$ in (iii) and $n=24$ in (iv). As a corollary of
Theorem \ref{main} and \cite[Theorem 2]{HalShu1985}, we obtain a classification of graphs that are
locally rank 3 where the local action is that of some subgroup of $S_n$ on the set $\tbinom{n}{2}$
of $2$-subsets of $\{1,\ldots,n\}$. In particular, this classification includes the locally
$\overline{T}_n$ graphs for $n\geq 5$.

\begin{cor}
\label{Hall Shult} Let $\Gamma$ be a connected non-complete graph with girth 3. For $n\geq 5$, let
$\mathscr{C}_n$ denote the class of groups $H\leq S_n$ such that $H$ is transitive of rank 3 on
$\tbinom{n}{2}$. Then the following are equivalent.
 \begin{enumerate}
 \item[(i)] There exists $G\leq \Aut(\Gamma)$ such that, for all $u\in V\Gamma$, the  action of $G_u$ on $\Gamma(u)$ is permutation isomorphic to  the action of some $H\in \mathscr{C}_n$ on
   $\tbinom{n}{2}$.
 \item[(ii)] $\Gamma$ is  $\overline{T}_{n+2}$,  or a graph from
   Theorem \ref{main} where $n\geq 5$, or one of four known graphs.
 \end{enumerate}
\end{cor}

The groups in the class $\mathscr{C}_n$ are listed in Theorem \ref{rank 3} and are determined using the classification of the finite simple groups.
The four exceptional graphs of Corollary \ref{Hall Shult} occur when $n=5$ or $6$. The exceptional
locally $\overline{T}_5$ graphs are the Conway-Smith graph and the commuting involutions graph of
the conjugacy class of the involutory Galois field automorphism in $\PSigmaL_2(25)$. The exceptional
locally $\overline{T}_6$ graphs are the complements of an elliptic quadric or a hyperplane in the
graph whose vertices are the non-zero vectors of a 6-dimensional $\mathbb{F}_2$-vector space, with
two vectors adjacent whenever they are perpendicular with respect to a given non-degenerate
symplectic form. More details may be found in \S \ref{HS}.

%%%%%%%%%%%%%%%%%%%%%%%%%%%%%%%%%%%%%%%%%%%%%%%%%%%
%
%	Rectagraphs
%
%%%%%%%%%%%%%%%%%%%%%%%%%%%%%%%%%%%%%%%%%%%%%%%%%%%

\subsection*{Rectagraphs} Locally triangular graphs are closely related to \textit{rectagraphs}, which  are connected triangle-free graphs in which every $2$-arc lies in a unique quadrangle. Indeed,
if $\Pi$ is a rectagraph with $a_2(\Pi)=0$ and $c_3(\Pi)=3$, then every connected component of
the distance 2 graph of $\Pi$ is locally triangular. (The parameters $a_i$ and $c_i$ are
the same as those defined for distance-regular graphs; see \S \ref{notation}.) Conversely, by
\cite[Proposition 4.3.9]{BroCohNeu1989}, every connected locally triangular graph is a halved graph
of some bipartite rectagraph $\Pi$ with $c_3(\Pi)=3$.

Rectagraphs were first named by Neumaier \cite{Neu1982}, though they were studied before this
\cite{Cam1975}. Distance-regular rectagraphs with certain parameters have been explored by various
authors \cite{Bro1983,RifHug1990}. More generally, rectagraphs are examples of
$(0,2)$-graphs, which have been classified for small valency \cite{Bro2006,BroOst2009}. Bipartite
rectagraphs also have links to geometry, for such graphs are precisely the incidence graphs of
semibiplanes, which are connected point-block incidence structures for which any two distinct points
lie in exactly 0 or 2 blocks and any two distinct blocks intersect in exactly 0 or 2 points. These
structures were first examined in \cite{Hug1978, Wil1980} and their study is an active area of
research.

Theorem \ref{main} is largely a consequence of our second main result, which concerns rectagraphs. 
We say that a group $H$ acting on a set $\Omega$ with $|\Omega|\geq 4$ is \textit{$4$-homogeneous} if $H$ acts transitively on the set of $4$-subsets of $\Omega$.

\begin{thm}
\label{main rect} 
Let $\Pi$ be a rectagraph  with $a_2(\Pi)=0$ and $c_3(\Pi)=3$. There exists $u\in V\Pi$ such that $|\Pi(u)|\geq 4$ and $\Aut(\Pi)_u$ is $4$-homogeneous on $\Pi(u)$ if and only if $\Pi$ is one of the following.
\begin{itemize}
\item[(i)] The $n$-cube $Q_n$ where $n\geq 4$.
\item[(ii)] The  folded  $n$-cube $\Box_n$ where  $n\geq 7$.
\item[(iii)] The bipartite double of the coset graph of the   binary Golay code $C_{23}$.
\item[(iv)] The  coset graph of the   binary Golay code $C_{23}$.
\item[(v)]  The coset graph of the extended binary Golay code $C_{24}$. 
\end{itemize}
\end{thm}

We remark that no assumptions of vertex-transitivity or distance-regularity are made in Theorems \ref{main} or \ref{main rect}, though vertex-transitivity is a consequence of the assumptions of Theorem \ref{main} (cf. Lemma \ref{ver tran}).

The graphs appearing in Theorem \ref{main rect} are described in \S \ref{Thm 1,2}. They are all
distance-transitive graphs of valency $n$, where $n=23$ in (iii)-(iv) and $n=24$ in (v). Moreover,
all of these graphs are coset graphs of linear codes in $\mathbb{F}_2^n$. Indeed, the $n$-cube is
the coset graph of the zero code, the folded $n$-cube is the coset graph of the repetition code, and
the bipartite double of the coset graph of the binary Golay code $C_{23}$ is isomorphic to the
coset graph of the even weight vectors in $C_{23}$. The graphs of Theorem \ref{main rect} are also
examples of affine $2$-arc transitive graphs. (Note that a $2$-arc transitive graph is precisely a
vertex-transitive locally $2$-arc transitive graph.) The primitive and bi-primitive affine $2$-arc
transitive graphs were classified in \cite{IvaPra1993}. The non-bipartite graphs of Theorem
\ref{main rect} are all primitive, while the bipartite graphs of Theorem \ref{main rect} are
bi-primitive except when $\Pi$ is the $n$-cube and $n$ is even.

A result of Cameron \cite[Theorem 4.5]{Cam1975} implies that if $\Pi$ is a vertex-transitive
rectagraph with $a_2(\Pi)=0$ and $c_3(\Pi)=3$ such that the action of $\Aut(\Pi)_u$  on $\Pi(u)$ is permutation isomorphic to the natural action of $S_n$ or $A_n$  on $[n]:=\{1,\ldots,n\}$, then either $\Pi$ is the $n$-cube, or $n\geq 7$ and
$\Pi$ is the folded $n$-cube. We obtain this result as a corollary of Theorem \ref{main rect}
without the assumption of vertex-transitivity.

\begin{cor}
\label{Cameron} Let $\Pi$ be a rectagraph with $a_2(\Pi)=0$ and $c_3(\Pi)=3$. There exists $u\in V\Pi$ such that the action of $\Aut(\Pi)_u$  on $\Pi(u)$ is permutation isomorphic to the natural action of $S_n$ or $A_n$ on $[n]$ if and only if $\Pi$ is the
$n$-cube, or $n\geq 7$ and $\Pi$ is the folded $n$-cube.
\end{cor}

In addition, Brouwer proved in \cite{Bro1983} that any distance-regular bipartite graph $\Pi$ with
parameters $c_i(\Pi)=i$ for all $i$ (and some additional assumptions) is the $n$-cube, the halved
$n$-cube or the coset graph of the extended binary Golay code. Replacing the distance-regularity
condition with a local symmetry condition, we obtain a similar result.

\begin{cor}
\label{Brouwer} Let $\Pi$ be a connected bipartite graph with $c_2(\Pi)=2$ and $c_3(\Pi)=3$. There exists $u\in V\Pi$ such that   $|\Pi(u)|\geq 5$ and $\Aut(\Pi)_u$ is $5$-transitive on $\Pi(u)$   if and only if $\Pi$ is the
$n$-cube where $n\geq 5$, or the folded $n$-cube where $n\geq 8$ and $n$ is even, or the coset graph of the extended
binary Golay code.
\end{cor}

The proof of Theorem \ref{main rect} proceeds as follows. If $\Pi$ is any rectagraph 
with $a_2(\Pi)=0$ and $c_3(\Pi)=3$, then for some $n$ there is a map $\pi:Q_n\to\Pi$ (called a covering) that
preserves the local structure of the $n$-cube \cite[\S 4.3B]{BroCohNeu1989}. By some unpublished
observations of Matsumoto \cite{Mat1991}, proved here in a slightly more general context,
there is a group $K^\pi$ of automorphisms of $Q_n$ associated to the covering $\pi$ that completely
determines the structure of $\Pi$ as a quotient of $Q_n$ (cf. Proposition \ref{K}). Often, this group $K^\pi$ turns out to be
a linear code in $\mathbb{F}_2^n$ as well as an $\mathbb{F}_2N_0^\pi$-module, where $N_0^\pi$ is some
subgroup of $S_n$ normalising $K^\pi$, in which case we can use coding theory and representation
theory to determine the group $K^\pi$ and therefore $\Pi$ itself.

%%%%%%%%%%%%%%%%%%%%%%%%%%%%%%%%%%%%%%%%%%%%%%%%%%%
%
%	Open problems
%
%%%%%%%%%%%%%%%%%%%%%%%%%%%%%%%%%%%%%%%%%%%%%%%%%%%

\subsection*{Open problems} One natural extension of our work would be to generalise Theorem
\ref{main} to the $q$-analogue of the triangular graph, the so-called \textit{Grassmann graph}, whose
vertices are the $2$-subspaces of an $\mathbb{F}_q$-vector space, with two $2$-subspaces
adjacent whenever their intersection has dimension one. Examples of locally Grassmann graphs include
the graph of alternating forms over $\mathbb{F}_2$ and the graph of quadratic forms over
$\mathbb{F}_2$ \cite{MunPasSch1993}. Surprisingly, it is claimed in \cite[Theorem 3]{KabMakPad2007}
that locally Grassmann graphs only exist when $q=2$. Unfortunately, no proof or reference to one is
given.

Moreover, since all of the graphs of Theorem \ref{main}, Corollary \ref{Hall Shult} and Theorem
\ref{main rect} turn out to be distance-transitive, it would be interesting to have a direct proof
of this fact.

%%%%%%%%%%%%%%%%%%%%%%%%%%%%%%%%%%%%%%%%%%%%%%%%%%%
%
%	Outline
%
%%%%%%%%%%%%%%%%%%%%%%%%%%%%%%%%%%%%%%%%%%%%%%%%%%%

\subsection*{Outline} In \S \ref{prelim}, we define some notation and give some information
concerning coverings, quotient graphs and binary linear codes, including the definitions of the
graphs of Theorems \ref{main} and \ref{main rect} and Corollary \ref{Hall Shult}. In \S \ref{rect},
we explore the structure of rectagraphs covered by $n$-cubes and use this information to prove
Theorem \ref{main rect} and Corollaries \ref{Cameron} and \ref{Brouwer}. In \S \ref{loc tri}, we
give some properties of locally triangular graphs, and in \S \ref{loc rank 3}, we examine locally
rank 3 graphs and then prove Theorem \ref{main} and Corollary \ref{Hall Shult}. Note that with the
exception of Corollary \ref{Cameron}, all of the main results of this paper depend on the
classification of the finite simple groups, as their proofs use the classification of the
multiply transitive permutation groups.

%%%%%%%%%%%%%%%%%%%%%%%%%%%%%%%%%%%%%%%%%%%%%%%%%%%
%
%	2. Preliminaries
%
%%%%%%%%%%%%%%%%%%%%%%%%%%%%%%%%%%%%%%%%%%%%%%%%%%%

\section{Preliminaries}
\label{cov quo}
\label{prelim}

Unless otherwise specified, all graphs in this paper are finite, undirected and simple (no multiple
edges or loops), all groups are finite, and all functions and actions are written on the right.
Basic graph theoretical terminology may be found in \cite{BroCohNeu1989}, and basic group
theoretical terminology may be found in \cite{Cam1999, Wil2009}. The notation used to denote the
finite simple groups is consistent with that of \cite{Wil2009}.

%%%%%%%%%%%%%%%%%%%%%%%%%%%%%%%%%%%%%%%%%%%%%%%%%%%
%
%	2.1 Notation and basic definitions
%
%%%%%%%%%%%%%%%%%%%%%%%%%%%%%%%%%%%%%%%%%%%%%%%%%%%

\subsection{Notation and basic definitions}
\label{notation}

Let $\mathbb{F}_2^n$ be the vector space of $n$-tuples over the field $\mathbb{F}_2=\{0,1\}$. The
\textit{weight} $\wt{u}$ of a vector $u\in \mathbb{F}_2^n$ is the number of non-zero coordinates in
$u$, and the \textit{Hamming distance} of $u,v\in\mathbb{F}_2^n$ is the number of coordinates at
which $u$ and $v$ differ, or equivalently, $\wt{u+v}$. For $1\leq i_1<\cdots<i_m\leq n$, let $e_{i_1,\ldots,i_m}$ denote the
vector of weight $m$ in $ \mathbb{F}_2^n$ whose $i_j$-th coordinate is 1 for $1\leq j\leq m$. Also, let $E_n$ denote the set of vectors in $\mathbb{F}_2^n$ with
even weight. Context permitting, we will write $0$ for the $n$-tuple $(0,\ldots,0)$ and $1$ for the
$n$-tuple $(1,\ldots,1)$. We write $\tbinom{n}{i}$ for the set of $i$-subsets of
$[n]:=\{1,\ldots,n\}$.

Let $G$ and $H$ be groups. Then $G\rtimes H$ denotes a semidirect product with normal subgroup $G$
and subgroup $H$, and $G.H$ denotes a group with normal subgroup $G$ and quotient $H$. Moreover,
$G\wr H$ denotes the wreath product $G^n\rtimes H$, where $H$ acts on $[n]$.  If $G$ acts on $\Omega$
and $\Delta:=\{\omega_1,\ldots,\omega_m\}\subseteq\Omega$, then we write  $G_\Delta$ for the setwise
stabiliser of $\Delta$ in $G$ and $G_{\omega_1,\ldots,\omega_m}$ for the pointwise stabiliser of
$\Delta$ in $G$. We also write $\omega_1^G$ for the orbit of $G$ containing $\omega_1$. The
\textit{induced permutation group} $G^\Omega$ of $G$ is defined to be the image of the permutation
representation $G\to \Sym(\Omega)$ and is isomorphic to $G/K$, where $K$ is the kernel of
the action of $G$ on $\Omega$. Hence $G^\Omega$ is a subgroup of $\Sym(\Omega)$. The symmetric group and alternating group on $n$ points are denoted by $S_n$ and $A_n$ respectively. For a field $F$, we denote the group algebra of $G$ over $F$ by $FG$, and if $G\leq S_n$, then the \textit{permutation module of $G$ over $F$} is the $FG$-module $F^n$ where $G$ acts by permuting coordinates.

If $G$ and $H$ are groups acting on $\Omega$ and $\Delta$ respectively, then $G$ and $H$ are
\textit{permutation isomorphic} if the action of $G$ or $H$ is faithful and there exists a group
isomorphism $\psi:G\to H$ and a bijection $\varphi:\Omega\to\Delta$ for which
$(\omega^g)\varphi=(\omega\varphi)^ {g\psi}$ for all $\omega\in\Omega$ and $g\in G$. Note that if
$G$ and $H$ are permutation isomorphic, then both groups must act faithfully. In particular, the
definition given is equivalent to the more standard definition of permutation isomorphism, which
requires that both $G$ and $H$ act faithfully.

Let $\Gamma$ be a graph. We write $V\Gamma$ for the vertex set of $\Gamma$, $E\Gamma$ for the edge
set of $\Gamma$, and $\Aut(\Gamma)$ for the automorphism group of $\Gamma$. We say that $\Gamma$ is
$G$-vertex-transitive (respectively $G$-edge-transitive) if $G\leq \Aut(\Gamma)$ and $G$ acts
transitively on $V\Gamma$ (respectively $E\Gamma$). If $X\subseteq V\Gamma$, then $[X]$ denotes the
subgraph of $\Gamma$ induced by $X$. The distance between $u,v\in V\Gamma$ is denoted by
$d_\Gamma(u,v)$, and for $u\in V\Gamma$ and any integer $i\geq 0$, we define $\Gamma_i(u):=\{v\in
V\Gamma:d_\Gamma(u,v)=i\}$. In particular, we write $\Gamma(u)$ for the neighbourhood $\Gamma_1(u)$.
For $u,v\in V\Gamma$ such that $d_\Gamma(u,v)=i$, let
\begin{align*}
c_i(u,v) & :=|\Gamma_{i-1}(u)\cap \Gamma(v)|, \\
a_i(u,v) & :=|\Gamma_{i}(u)\cap \Gamma(v)|. 
\end{align*}
We write $c_i(\Gamma)$ (respectively $a_i(\Gamma)$) whenever $c_i(u,v)$ (respectively $a_i(u,v)$)
does not depend on the choice of $u$ and $v$, and we omit the $\Gamma$ when context permits. Note that if $\Gamma$ is bipartite, then $a_i=0$ for all $i$.  We
write $\overline{\Gamma}$ for the complement of $\Gamma$. The complete graph on $n$ vertices is
denoted by $K_n$, and the complete multipartite graph with $n$ parts of size $m$ is denoted by
$K_{n[m]}$.

The \textit{distance 2 graph} $\Gamma_2$ of $\Gamma$ has vertex set $V\Gamma$, where two vertices
are adjacent whenever their distance in $\Gamma$ is 2. If $\Gamma$ is connected but not bipartite,
then $\Gamma_2$ is connected, and if $\Gamma$ is connected and bipartite, then $\Gamma_2$ has
exactly two connected components; these are called the \textit{halved graphs} of $\Gamma$. We write
$\tfrac{1}{2}\Gamma$ for a halved graph of $\Gamma$ whenever the halved graphs of $\Gamma$ are
isomorphic.

The \textit{bipartite double} $\Gamma.2$ of $\Gamma$ has vertex set $V\Gamma\times \mathbb{F}_2$,
where vertices $(u,x)$ and $(v,y)$ are adjacent whenever $u$ and $v$ are adjacent in $\Gamma$ and
$x\neq y$. The graph $\Gamma.2$ is bipartite, and it is connected if and only if $\Gamma$ is
connected but not bipartite.

 %%%%%%%%%%%%%%%%%%%%%%%%%%%%%%%%%%%%%%%%%%%%%%%%%%%
%
%	2.2 Covering maps
%
%%%%%%%%%%%%%%%%%%%%%%%%%%%%%%%%%%%%%%%%%%%%%%%%%%%

\subsection{Covering maps}
\label{covering maps}

Let $\Gamma$ and $\Pi$ be graphs. A map $\pi:\Gamma\to\Pi$ is a \textit{local bijection} if $\pi$
induces a bijection from $\Gamma(x)$ onto $\Pi(x\pi)$ for all $x\in V\Gamma$. Note that a local
bijection is also called a \textit{local isomorphism}, but we prefer the former term since the
induced neighbourhood graphs need not be isomorphic in general. A surjective local bijection is a
\textit{covering}. Whenever a covering $\pi:\Gamma\to\Pi$ exists, we say that $\Pi$ is
\textit{covered} by $\Gamma$.

Here are some basic but important properties of local bijections and coverings.

\begin{lemma}
\label{onto}
Let $\Gamma$ and $\Pi$ be graphs. If $\Pi$ is connected, then any local bijection  $\pi:\Gamma\to\Pi$ is a covering.
\end{lemma}

\begin{proof}
If $u\in V\Pi$ is adjacent to $x\pi$ for some $x\in V\Gamma$, then $u\in\Pi(x\pi)=\Gamma(x)\pi$, and
so $u=y\pi$ for some $y\in\Gamma(x)$. Since $\Pi$ is connected, $\pi$ is a covering.
\end{proof}

Note that there exist local bijections that are not coverings. For example, there is a local
bijection from $K_2.2$ to itself whose image is $K_2$.

\begin{lemma}
\label{cover}
Let  $\Gamma$ and $\Pi$ be graphs, and let  $\pi:\Gamma\to\Pi$ be a covering. Then the following hold.
\begin{itemize}
\item[(i)] If $u_1,u_2\in V\Pi$ are adjacent, then for every $x_1\in u_1\pi^{-1}$, there exists a
  unique $x_2\in u_2\pi^{-1}$ such that $x_1$ and $x_2$ are adjacent.
\item[(ii)] If $\Pi$ is connected, then $|u\pi^{-1}|=|v\pi^{-1}|$ for all $u,v\in V\Pi$.
\end{itemize}
\end{lemma}

\begin{proof}
(i) Let $x_1\in u_1\pi^{-1}$. Then $u_2\in \Pi(x_1\pi)$. Since $\pi$ is a local bijection, there
exists a unique $x_2\in \Gamma(x_1)$ such that $u_2=x_2\pi$ (and hence $x_2\in u_2\pi^{-1}$).

(ii) Follows from (i).
\end{proof}

Let $\Gamma$ and $\Pi$ be graphs, and let $\pi:\Gamma\to\Pi$ be a covering. If $g\in \Aut(\Gamma)$,
then $g:\Gamma\to\Gamma$, so we may compose the functions $g$ and $\pi$ to obtain a new covering
$g\pi :\Gamma\to\Pi$. Since $g\pi=\pi$ if and only if $\pi=g^{-1}\pi$, it follows that
$\{g\in\Aut(\Gamma):g\pi=\pi\}$ forms a subgroup of $\Aut(\Gamma)$. We denote this subgroup by $K^\pi$.
We will see in \S \ref{rect} that $K^\pi$ is fundamental to the study of many rectagraphs. Here is
one useful property of this group.

\begin{lemma}
\label{ker} Let $\Gamma$ and $\Pi$ be graphs where $\Gamma$ is connected, and let $\pi:\Gamma\to\Pi$
be a covering. Then $K^\pi_x=1$ for all $x\in V\Gamma$.
\end{lemma}

\begin{proof}
 Let $x\in V\Gamma$ and $y\in \Gamma(x)$. Set $u:=x\pi$ and $v:=y\pi$. Then $v\in\Pi(u)$. Let $g\in
 K^\pi_x$. Then $y^g\pi=y(g\pi)=y\pi$, so $y^g\in v\pi^{-1}$, but $x^g=x$, so $y$ and $y^g$ are both
 neighbours of $x$ in $v\pi^{-1}$. Hence $y=y^g$ by Lemma \ref{cover} (i). Since $\Gamma$ is
 connected, it follows that $g$ fixes $V\Gamma$ pointwise, and so $g=1$.
\end{proof}

%%%%%%%%%%%%%%%%%%%%%%%%%%%%%%%%%%%%%%%%%%%%%%%%%%%
%
%	2.3 Quotient graphs
%
%%%%%%%%%%%%%%%%%%%%%%%%%%%%%%%%%%%%%%%%%%%%%%%%%%%

\subsection{Quotient graphs}
\label{Q graphs}

For a graph $\Gamma$ and a partition $\mathcal{B}$ of $V\Gamma$, the \textit{quotient graph}
$\Gamma_\mathcal{B}$ is the graph with vertex set $\mathcal{B}$, where  $B_1,B_2\in\mathcal{B}$ are adjacent whenever there exists $x_1\in B_1$ and $x_2\in B_2$ such that $x_1$ and
$x_2$ are adjacent in $\Gamma$. Note that $\Gamma_\mathcal{B}$ can contain loops. 

Any graph  covered by $\Gamma$ is naturally isomorphic to a quotient graph of $\Gamma$, for if $\Pi$ is a graph and $\pi:\Gamma\to\Pi$ is a covering, then $\Pi\simeq \Gamma_\mathcal{B}$ where  $\mathcal{B}$ is the set of preimages of $\pi$ in $V\Gamma$. On the other hand, given a graph $\Gamma$ and a partition $\mathcal{B}$ of $V\Gamma$, there is a natural surjective map $\pi:\Gamma\to\Gamma_{\mathcal{B}}$ sending a vertex in $\Gamma$ to the part in $\mathcal{B}$ containing it, and this map  is a covering precisely when the  following two conditions hold: (i)  no $B\in\mathcal{B}$ contains an edge, and (ii) if  $B_1,B_2\in V\Gamma_\mathcal{B}$ are adjacent, then for every $x_1\in B_1$, there   exists a unique $x_2\in B_2$ such that $x_1$ and $x_2$ are adjacent (cf. Lemma \ref{cover}).

One useful way of defining partitions for graph quotients is to use the orbits of a normal subgroup
of a group of automorphisms, for we retain some control of the automorphism group, the valency
and the local action of the quotient.

Let $M\leq \Aut(\Gamma)$ (where $M$ is not necessarily normal in $\Aut(\Gamma)$), and let
$\mathcal{B}$ be the set of orbits of $M$ on $V\Gamma$. We say that $\Gamma_{\mathcal{B}}$ is a
\textit{normal quotient} of $\Gamma$ and write $\Gamma_M$ for $\Gamma_{\mathcal{B}}$. If $M\unlhd
G\leq \Aut(\Gamma)$, then $G$ acts naturally on $\mathcal{B}$ by $(x^M)^g:=(x^g)^M$ for all $x\in
V\Gamma$ and $g\in G$. Moreover, this action preserves adjacency, so $G/K\leq \Aut(\Gamma_M)$, where
$K$ is the kernel of the action of $G$ on $\mathcal{B}$. Note that $M\leq K$, but $M\neq K$ in
general.

Recall from \S \ref{covering maps} that if $\pi:\Gamma\to\Pi$ is a covering, then $K^\pi=\{g\in \Aut(\Gamma):g\pi=\pi\}$, and
let $N^\pi$ denote the normaliser of $K^\pi$ in $\Aut(\Gamma)$. In the following, we begin to see
the importance of the group $K^\pi$.

\begin{lemma}
\label{K cover} Let $\Gamma$ be a connected graph, and let $M\leq \Aut(\Gamma)$. If the natural map
$\pi:\Gamma \to \Gamma_M$ is a covering, then the following hold.
\begin{itemize}
 \item[(i)] $K^\pi=M$, and $K^\pi$ is the kernel of the action of $N^\pi$ on $V\Gamma_M$.
\item[(ii)] $N^\pi/K^\pi\leq \Aut(\Gamma_M)$.
\item[(iii)] $N^\pi_x\simeq (N^\pi/K^\pi)_{x\pi}$ for all $x\in V\Gamma$.
\end{itemize}
\end{lemma}

\begin{proof}
 Clearly $M\leq K^\pi$. If $g\in K^\pi$ and $x\in V\Gamma$, then $(x^g)^M=x(g\pi)=x\pi=x^M$, so $x^g=x^m$ for
 some $m\in M$. Then $gm^{-1}\in K^\pi_x$, but $K^\pi_x=1$ by Lemma \ref{ker}, so $g\in M$. Thus
 $M=K^\pi$. In particular, $N^\pi$ acts on $V\Gamma_M$, and $K^\pi$ is contained in the kernel of
 this action. If $g\in N^\pi$ and $(x^M)^g=x^M$ for all $x\in V\Gamma$, then $x(g\pi)=x^g\pi=x\pi$ for all
 $x\in V\Gamma$, so $g\in K^\pi$. Thus $K^\pi$ is the kernel of the action of $N^\pi$ on
 $V\Gamma_M$, proving (i) and (ii).

Let $x\in V\Gamma$. There is a group homomorphism $\psi: N^\pi_x\to (N^\pi/K^\pi)_{x\pi}$ defined by
$g\mapsto K^\pi g$ for all $g\in N^\pi_x$, and it is injective since $N^\pi_x\cap K^\pi=1$ by Lemma
\ref{ker}. To see that $\psi$ is surjective, let $K^\pi g\in (N^\pi/K^\pi)_{x\pi}$. Then $x^g$ and
$x$ are in the same $K^\pi$-orbit, so there exists $k\in K^\pi$ such that $x^{gk}=x$. Thus $gk\in
N^\pi_x$ and $(gk)\psi=K^\pi gk = K^\pi(gkg^{-1})g= K^\pi g$, as desired.
\end{proof}

%%%%%%%%%%%%%%%%%%%%%%%%%%%%%%%%%%%%%%%%%%%%%%%%%%%
%
%	2.4 Coset graphs of binary linear codes
%
%%%%%%%%%%%%%%%%%%%%%%%%%%%%%%%%%%%%%%%%%%%%%%%%%%%

\subsection{Coset graphs of binary linear codes}
\label{coset graphs}

The \textit{$n$-cube} $Q_n$ is defined to be the graph with vertex set $\mathbb{F}_2^n$, where two
$n$-tuples are adjacent whenever their Hamming distance is 1. The $n$-cube is a connected regular bipartite graph of valency $n$
with parts the sets of vectors of even and odd weight. Its automorphism group is $\mathbb{F}_2\wr
S_n=\mathbb{F}_2^n\rtimes S_n$, where $\mathbb{F}_2^n$ acts on $VQ_n=\mathbb{F}_2^n$ by translation,
and $S_n$ acts by permuting coordinates.

In this paper, we will be interested in normal quotients of $Q_n$, particularly those that are
formed using additive subgroups of $\mathbb{F}_2^n$. If $C$ is such a group, we define the
\textit{coset graph} of $C$, denoted by $\Gamma(C)$, to be the normal quotient $(Q_n)_C$. This graph
is so named because the orbits of $C$ on $\mathbb{F}_2^n$ are the cosets of $C$ in $\mathbb{F}_2^n$.
Note that $\Gamma(C)$ is loop-free if and only if $e_i\notin C$ for all $i\leq n$.

The additive subgroups of $\mathbb{F}_2^n$ are precisely the binary linear codes, for which we now
give some basic terminology. A \textit{binary linear code} $C$ is defined to be a subspace of
$\mathbb{F}_2^n$. The vectors in $C$ are called \textit{codewords}. The \textit{minimum distance} of
$C$ is defined to be $\infty$ when $C=\{0\}$, and the minimum Hamming distance between distinct
codewords in $C$ otherwise. Note that the minimum distance of a non-zero code is precisely the
minimum weight of the non-zero codewords. We say that $C$ is an \textit{$[n,r,d]$-code}, where $r$
is the dimension of $C$ and $d$ is the minimum distance of $C$. If every codeword in $C$ has even
weight, then $C$ is an \textit{even} code. The \textit{automorphism group} of $C$, denoted by $
\Aut(C)$, is defined to be the stabiliser of $C$ in $S_n$, where $S_n$ acts on $\mathbb{F}_2^n$ by
permuting coordinates.

The following describes some elementary but important properties of binary linear codes and the coset graphs formed from them.

\begin{lemma}
\label{code}
Let $C$ be a binary linear $[n,r,d]$-code where  $d\geq 2$. Let $\Pi:=\Gamma(C)$. Then the following hold.
\begin{itemize}
 \item[(i)] The graph $\Pi$ is bipartite if and only if $C$ is an even code.
\item[(ii)] If $\Pi$ is bipartite, then the  halved graphs of  $\Pi$ are isomorphic. 
\item[(iii)] If $C$ is not an even code, then $ \Pi.2\simeq \Gamma(C\cap E_n)$.
\item[(iv)] If $d\geq 3$, then the natural map $\pi: Q_n \to \Pi$ is a covering. 
\item[(v)] If $d\geq 5$, then $a_1(\Pi)=0$ and $c_2(\Pi)=2$. 
\item[(vi)] If $d\geq 7$, then $a_2(\Pi)=0$ and $c_3(\Pi)=3$.
\end{itemize}
 \end{lemma}

\begin{proof}
(i) If $C$ is even, then the weights of vectors in a coset of $C$ have the same parity. Edges in
$Q_n$ only occur between vectors whose weights have different parity, so $\Pi$ is bipartite.

If $C$ is not even, then we may choose $c\in C$ of minimal odd weight. There exists a path
$(x_0,\ldots,x_m)$ in $Q_n$ where $x_0=0$, $x_m=c$, and $m=\wt{c}$. The minimality of $c$ and $m\geq
d\geq 2$ then imply that $(x_0+C,\ldots,x_{m}+C)$ is an odd cycle, so $\Pi$ is not bipartite.

(ii) Let $\Gamma$ and $\Sigma$ be the halved graphs of $\Pi$ (if they exist). By the proof of (i),
we may define a map $\varphi:\Gamma\to \Sigma$ by $x+C\mapsto x+e_1+C$ for all $x+C\in V\Gamma$.
It is routine to verify that $\varphi$ is a graph isomorphism.

(iii) Define a map $\varphi:\Gamma(C\cap E_n)\to \Pi.2$ by $x+C\cap E_n\mapsto (x+C,\wt{x}\mod 2)$ for all $x\in VQ_n$. It is routine to verify that $\varphi$ is a well-defined graph isomorphism.

(iv) This follows from the fact that $\Pi(x+C)=\{x+e_i+C:i\in [n]\}$ for all $x\in VQ_n$.

(v) This follows from the structure of $\Pi(x+C)$ observed in (iv) together with
$\Pi_2(x+C)=\{x+e_{i,j}+C:\{i,j\}\in \tbinom{n}{2}\}$ for all $x\in VQ_n$.

(vi) This follows from the observations in (iv) and (v) together with
$\Pi_3(x+C)=\{x+e_{i,j,k}+C:\{i,j,k\}\in \tbinom{n}{3}\}$ for all $x\in VQ_n$.
\end{proof}

Thus whenever $\Gamma(C)$ is bipartite, its halved graphs are isomorphic, and so we may refer to
\textit{the} halved graph of $\Gamma(C)$. In particular, $C\leq E_n$, and we may
assume that the halved graph of $\Gamma(C)$ has vertex set $E_n/C$. Note that if $C$ is not an even code, then there is a natural graph isomorphism between $\tfrac{1}{2}\Gamma(C\cap E_n)$ and the distance 2 graph $\Gamma(C)_2$.

%%%%%%%%%%%%%%%%%%%%%%%%%%%%%%%%%%%%%%%%%%%%%%%%%%%
%
%	2.5 Graphs of Theorems 1.1 and 1.3
%
%%%%%%%%%%%%%%%%%%%%%%%%%%%%%%%%%%%%%%%%%%%%%%%%%%%

\subsection{Graphs of Theorems \ref{main} and  \ref{main rect}}
\label{Thm 1,2}

The graphs of Theorem \ref{main rect} are all coset graphs of binary linear codes, and the graphs of
Theorem \ref{main} are all halved graphs of coset graphs of even binary linear codes. There are five
relevant binary linear codes for these theorems, two of which are elementary.

The most basic code is the \textit{zero code} $\{0\}$, whose coset graph is simply $Q_n$. It is an
even binary linear $[n,0,\infty]$-code with automorphism group $S_n$. The $n$-cube is
distance-transitive with automorphism group $2^n\rtimes S_n$ and vertex stabiliser $S_n$. The
automorphism group of the halved $n$-cube is $2^{n-1}\rtimes S_n$ for $n\geq 5$, $2^4\rtimes S_4$
for $n=4$, and $S_{4}$ for $n=3$ \cite[p. 265]{BroCohNeu1989}. It has vertex stabiliser $S_n$ when
$n\neq 4$, and vertex stabiliser $S_4\times C_2$ when $n=4$. The halved $n$-cube is also
distance-transitive, as is any halved graph of a bipartite distance-transitive graph \cite[Theorem
4.1.10]{BroCohNeu1989}.

The other basic code is the \textit{repetition code} $\{0,1\}$, whose coset graph $\Box_n$ is called
the \textit{folded $n$-cube}. It is a binary linear $[n,1,n]$-code with automorphism group $S_n$,
and it is an even code if and only if $n$ is even. The folded $n$-cube is distance-transitive with
automorphism group $2^{n-1}\rtimes S_n$ and vertex stabiliser $S_n$ for $n\geq 5$. When $n$ is even
and $n\geq 8$, the halved folded $n$-cube is distance-transitive with automorphism group
$2^{n-2}\rtimes S_n$ and vertex stabiliser $S_n$ \cite[p. 265]{BroCohNeu1989}.

The remaining codes are Golay codes. The \textit{extended binary Golay code} $C_{24}$ is a binary
linear $[24,12,8]$-code. There are various methods for constructing this code; see \cite{Wil2009},
for example. Its automorphism group is the Mathieu group $M_{24}$, and it is an even code. The graph
$\Gamma(C_{24})$ is distance-transitive with automorphism group $2^{12}\rtimes M_{24}$ and vertex
stabiliser $M_{24}$ \cite[Theorem 11.3.2]{BroCohNeu1989}. The halved graph of $\Gamma(C_{24})$ is
distance-transitive with automorphism group $2^{11}\rtimes M_{24}$ and vertex stabiliser $M_{24}$
(cf. Lemma \ref{rank 3 code}).

 If we remove one fixed coordinate from every codeword in the extended binary Golay code, then we
 obtain the \textit{binary Golay code} $C_{23}$, which is a binary linear $[23,12,7]$-code. Its
 automorphism group is the Mathieu group $M_{23}$, the stabiliser in $M_{24}$ of a point, and it is
 not an even code. The graph $\Gamma(C_{23})$ is distance-transitive with automorphism group
 $2^{11}\rtimes M_{23}$ and vertex stabiliser $M_{23}$ \cite[Theorem 11.3.4]{BroCohNeu1989}.

Lastly, the set of vectors in $C_{23}$ with even weight forms a binary linear $[23,11,8]$-code. Its
automorphism group is $M_{23}$ since automorphisms of codes preserve weight, and its coset graph is
isomorphic to the bipartite double $\Gamma(C_{23}).2$ by Lemma \ref{code} (iii). The graph $\Gamma(C_{23}).2$ is
distance-transitive with automorphism group $2^{12}\rtimes M_{23}$ and vertex stabiliser $M_{23}$ \cite[p. 362]{BroCohNeu1989}. 
The halved graph of $\Gamma(C_{23}).2$ is distance-transitive with automorphism group $2^{11}\rtimes
M_{23}$ and vertex stabiliser $M_{23}$ (cf. Lemma \ref{rank 3 code}).

 Note that if $C$ is one of
the linear codes in $\mathbb{F}_2^n$ defined above, or any binary linear code with $d\geq 5$, then
$\Aut(\Gamma(C))$ can be described uniformly as $(\mathbb{F}_2^n/C)\rtimes \Aut(C)$ (cf. Lemma
\ref{code aut}), where the vertex set $\mathbb{F}_2^n/C$ acts by translation and $\Aut(C)$ acts by
permuting coordinates.  Moreover, when $C$ is an even code and $d\geq 5$, the automorphism group of the halved graph of $\Gamma(C)$
 always contains $(E_n/C)\rtimes \Aut(C)$ (cf. Lemma \ref{aut2}), and is in fact equal to this group
 for $d\geq 7$ except when $C$ is the zero code and $n=4$ (cf. Lemma \ref{rank 3 code}).

%%%%%%%%%%%%%%%%%%%%%%%%%%%%%%%%%%%%%%%%%%%%%%%%%%%
%
%	2.6 Graphs of Corollary 1.2
%
%%%%%%%%%%%%%%%%%%%%%%%%%%%%%%%%%%%%%%%%%%%%%%%%%%%

\subsection{Graphs of Corollary \ref{Hall Shult}}
\label{HS}

The graphs of Corollary \ref{Hall Shult} that are not described in Theorem \ref{main} are precisely
the locally $\overline{T}_n$ graphs for $n\geq 5$, classified by Hall and Shult \cite[Theorem
2]{HalShu1985}. The graph $\overline{T}_{n+2}$ is itself locally $\overline{T}_n$ for $n\geq 5$.
This graph is strongly regular and distance-transitive with automorphism group $S_{n+2}$ and vertex
stabiliser $S_n\times C_2$, where $C_2$ fixes the neighbourhood of the vertex pointwise.

The remaining locally $\overline{T}_n$ graphs only occur when $n=5$ or 6. When $n=5$, the graph
$\overline{T}_5$ is isomorphic to the Petersen graph. Graphs that are locally Petersen were
classified by Hall \cite{Hall1980}. Besides $\overline{T}_7$, there are two such graphs, both of
which are commuting involutions graphs. Given a group $G$ and a conjugacy class $\mathscr{C}$ of involutions in $G$, the \textit{commuting
involutions graph of $\mathscr{C}$ in $G$} has vertex set $\mathscr{C}$, where two
(distinct) involutions are adjacent whenever they commute.
 For example, $\overline{T}_n$ is 
the commuting involutions graph of the class of transpositions in $S_n$.

One of the exceptional locally Petersen graphs is the \textit{Conway-Smith graph}, which is the commuting graph of transposition preimages in the group $3. S_7$. It is distance-transitive
and has automorphism group $3. S_7$ and vertex stabiliser $S_5\times C_2$, where $C_2$ fixes the
neighbourhood of the vertex pointwise. See \cite[Theorem 13.2.3]{BroCohNeu1989} for more details.

The other exceptional locally Petersen graph is the commuting involutions graph of the conjugacy class of the involutory Galois field
automorphism in the group $\PSigmaL_2(25)$. It is distance-transitive and has automorphism group
$\PSigmaL_2(25)$ and vertex stabiliser $S_5\times C_2$, where $C_2$ fixes the neighbourhood of the
vertex pointwise. See \cite[Proposition 12.2.2]{BroCohNeu1989} for more details.

When $n=6$, there are again two locally $\overline{T}_6$ graphs besides $\overline{T}_8$. Let
$\Sp_{2n}(2)$ denote the graph whose vertices are the non-zero vectors of a $2n$-dimensional
$\mathbb{F}_2$-vector space, with two vectors adjacent whenever they are perpendicular with respect
to a given non-degenerate symplectic form. The locally $\Sp_{2n}(2)$ graphs were classified in
\cite[Theorem 5]{HalShu1985}, and since $\overline{T}_6\simeq \Sp_4(2)$, this classification
determines the locally $\overline{T}_6$ graphs. These graphs all arise as subgraphs of $\Sp_6(2)$,
including $\overline{T}_8$, which is isomorphic to the complement of a hyperbolic quadric in
$\Sp_6(2)$.

One of the exceptional locally $\overline{T}_6$ graphs is the complement of an elliptic quadric in $\Sp_6(2)$.
This graph has 36 vertices and diameter
2. Using {\sf GAP} \cite{GAP4,FinInG,Grape}, we determined that it is distance-transitive with the following distance distribution diagram.

{\centering \includegraphics{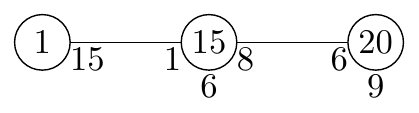} \par} \noindent We also determined that it has automorphism group
$\POmega_6^-(2)\rtimes 2$ and vertex stabiliser $S_6\times C_2$, where $C_2$ fixes the neighbourhood of the
vertex pointwise.

The other exceptional locally $\overline{T}_6$ graph is the complement of a hyperplane in $\Sp_6(2)$. This graph has 32 vertices and diameter 3. Using {\sf GAP} \cite{GAP4,FinInG,Grape}, we determined that it  is distance-transitive with the
following distance distribution diagram.

{\centering \includegraphics{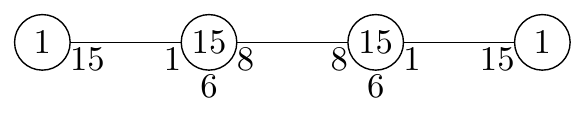} \par} \noindent We also determined that it has automorphism group
$2^5\rtimes S_6$ and vertex stabiliser $S_6$. Interestingly, this graph is isomorphic to a connected
component in the distance 4 graph of $Q_6$ (where $x,y\in VQ_6$ are adjacent whenever
$d_{Q_6}(x,y)=4$). It is also isomorphic to the subgraph of the complement of $\tfrac{1}{2}Q_6$
where the edge between $x$ and $x+1$ is removed for all $x\in V(\tfrac{1}{2}Q_6)$.

%%%%%%%%%%%%%%%%%%%%%%%%%%%%%%%%%%%%%%%%%%%%%%%%%%%
%
%	Rectagraphs
%
%%%%%%%%%%%%%%%%%%%%%%%%%%%%%%%%%%%%%%%%%%%%%%%%%%%

\section{Rectagraphs} \label{rect}

Recall that a \textit{rectagraph} is a connected triangle-free graph in which every $2$-arc
determines a unique quadrangle. Equivalently, a rectagraph is a connected graph with $a_1=0$ and
$c_2=2$. By \cite[Proposition 1.1.2]{BroCohNeu1989}, every rectagraph is regular.

The most basic example of a rectagraph is also the most important one: the $n$-cube $Q_n$. The
$n$-cubes are extremal rectagraphs, for a rectagraph $\Pi$ of valency $n$ has at most $2^n$ vertices
with equality if and only if $\Pi$ is the $n$-cube \cite[Proposition 1.13.1]{BroCohNeu1989}.

Other examples of rectagraphs include coset graphs of binary linear codes with minimum distance at
least five, all of which are covered by the $n$-cube (cf. Lemma \ref{code} (iv)-(v)). We begin
with a fundamental generalisation of this result for a certain class of rectagraphs. Note that this
result can be proved under more general assumptions than $a_2=0$ and $c_3=3$, but these suffice for
our purposes.

\begin{lemma}[\cite{BroCohNeu1989}]
\label{pi} Let $\Pi$ be a rectagraph of valency $n$ where $a_2=0$ and $c_3=3$. For any $u\in V\Pi$
with neighbours $u_1,\ldots,u_n$, there exists a covering $\pi:Q_n\to\Pi$ such that $0\pi=u$ and
$e_i\pi=u_i$ for all $i\leq n$.
\end{lemma}

\begin{proof}
By \cite[Lemma 4.3.5]{BroCohNeu1989} and the proof of \cite[Proposition 4.3.6]{BroCohNeu1989}, there
exists a map $\pi:Q_n\to\Pi$ for which $0\pi=u$ and $e_i\pi=u_i$ for all $i\leq n$, and also if
$d_{Q_n}(x,y)\leq 2$ for $x,y\in VQ_n$, then $d_{Q_n}(x,y)=d_\Pi(x\pi,y\pi)$. Since $Q_n$ and $\Pi$
have valency $n$, it follows that $\pi$ is a local bijection, and so $\pi$ is a covering by Lemma
\ref{onto}.
\end{proof}

In fact, whenever a covering of a rectagraph by an $n$-cube exists, it is essentially unique.

\begin{lemma}
\label{unique} Let $\Pi$ be a rectagraph. Let $\pi:Q_n\to \Pi$ and $\theta:Q_n\to \Pi$ be coverings.
If $0\pi=0\theta$ and $e_i\pi=e_i\theta$ for all $i\leq n$, then $\pi=\theta$.
\end{lemma}

\begin{proof}
 We prove that $x\pi=x\theta$ for all $x\in VQ_n$ by induction on $\wt{x}$. If $\wt{x}\leq 1$, then
 $x\pi=x\theta$ by assumption, so we may assume that $\wt{x}\geq 2$. Let $i$ and $j$ be non-zero
 coordinates of $x$. By induction, $(x+e_i)\pi=(x+e_i)\theta$, $(x+e_j)\pi=(x+e_j)\theta$ and
 $(x+e_i+e_j)\pi=(x+e_i +e_j)\theta$. But any covering maps quadrangles to quadrangles, so $x\pi$
 and $x\theta$ are both vertices of $\Pi$ that are adjacent to $(x+e_i)\pi$ and $(x+e_j)\pi$ but
 distinct from $(x+e_i+e_j)\pi$. Thus $x\pi=x\theta$.
\end{proof}

The next two results are slight generalisations of some unpublished work by Matsumoto
\cite{Mat1991}, who observed that a certain type of rectagraph covered by $Q_n$ must also be a normal
quotient of $Q_n$ \cite[Proposition 2]{Mat1991} and then determined the automorphism group of this
normal quotient \cite[Lemma 5]{Mat1991}.

First we recall some notation. Let $\Gamma$ and $\Pi$ be graphs, and let $\pi:\Gamma\to\Pi$ be a
covering. Recall from \S \ref{covering maps} that $K^\pi=\{g\in \Aut(\Gamma):g\pi=\pi\}\leq
\Aut(\Gamma)$, and recall from \S \ref{Q graphs} that $N^\pi$ denotes the normaliser of $K^\pi$ in $\Aut(\Gamma)$. The groups
$K^\pi$ and $N^\pi$ play a vital role in the study of rectagraphs, as we now see. Note that 
 we write elements
of $\Aut(Q_n)=\mathbb{F}_2\wr S_n$ in the form $(x,\sigma)$ where $x\in \mathbb{F}_2^n$ and $\sigma\in S_n$.

\begin{lemma}
\label{preimage K} Let $\Pi$ be a rectagraph and $\pi:Q_n\to\Pi$ a covering. Then $v\pi^{-1}$ is a
regular $K^\pi$-orbit for all $v\in V\Pi$.
\end{lemma}

\begin{proof}
 Let $u:=0\pi$. For $x\in v\pi^{-1}$ and $g\in K^\pi$, we have $v=x\pi=x(g\pi)=x^g\pi$, so $v\pi^{-1}$ is a
 $K^\pi$-invariant set. Since $|v\pi^{-1}|=|u\pi^{-1}|$ by Lemma \ref{cover} (ii) and $K^\pi_x=1$
 for all $x\in VQ_n$ by Lemma \ref{ker}, it suffices to prove that $u\pi^{-1}$ is a $K^\pi$-orbit.
 If $\{0\}=u\pi^{-1}$, then $K^\pi=K^\pi_0=1$, so $u\pi^{-1}$ is a $K^\pi$-orbit, as desired.
 Otherwise, let $0\neq y\in u\pi^{-1}$. Then $\{e_i\pi:1\leq i\leq
 n\}=\Pi(0\pi)=\Pi(y\pi)=\{(y+e_i)\pi:1\leq i\leq n\},$ so there exists $\sigma\in S_n$ for which
 $e_i\pi=(y+e_{i^\sigma})\pi$ for all $i\leq n$. Let $g:=(y^{\sigma^{-1}},\sigma)\in
 \mathbb{F}_2^n\rtimes S_n=\Aut(Q_n)$. Then $0^g=y$ and $e_i^g=y+e_{i^\sigma}$ for all $i\leq n$. This
 implies that $g\pi$ and $\pi$ are coverings that agree on $\{0\}\cup Q_n(0)$, so $g\pi=\pi$ by
 Lemma \ref{unique}. Hence $g\in K^\pi$. Since $0^g=y$, it follows that $u\pi^{-1}$ is a
 $K^\pi$-orbit.
\end{proof}

Thus any rectagraph covered by an $n$-cube is a normal quotient of $Q_n$. We record this fact in Proposition \ref{K} below. Moreover, using Lemma
\ref{preimage K}, we can describe the automorphism group of a rectagraph covered by an $n$-cube in
terms of groups related to the covering.

\begin{prop}
\label{K} Let $\Pi$ be a rectagraph and $\pi:Q_n\to\Pi$ a covering. Then $\Pi\simeq (Q_n)_{K^\pi} $
and $ \Aut(\Pi)\simeq N^\pi/K^\pi$.
\end{prop}

\begin{proof}
By Lemma \ref{preimage K}, the set $v\pi^{-1}$ is a $K^\pi$-orbit for all $v\in V\Pi$, so there is a graph isomorphism $\alpha:(Q_n)_{K^\pi}\to\Pi$ defined by
$x^{K^\pi}\mapsto x\pi$ for all $x\in VQ_n$. Let $\pi':Q_n\to (Q_n)_{K^\pi}$ be the natural map.
Then $\pi'\alpha=\pi$, so $\pi'$ is a covering. Thus $K^\pi=K^{\pi'}$ by Lemma \ref{K cover}, and
since we then have $N^\pi=N^{\pi'}$, we may assume that $\Pi=(Q_n)_{K^\pi}$ and $\pi=\pi'$.

 By Lemma \ref{K cover}, we have $N^\pi/K^\pi\leq \Aut(\Pi)$. Let $h\in \Aut(\Pi)$ and define
 $\theta:=\pi h^{-1}$. Let $y\in u\theta^{-1}$ where $u:=0\pi$. Since $\theta:Q_n\to \Pi$ is a
 covering, it induces a bijection from $Q_n(y)$ onto $\Pi(u)$. Since $a_2(Q_n)=0$ and $c_3(Q_n)=3$,  Lemma \ref{pi}  implies that there exists a
 covering $g:Q_n\to Q_n$ for which $0 g=y$ and $e_ig=e_i\pi\theta^{-1}$ for all $ i\leq n$. Then
 $\pi$ and $g\theta$ agree on $\{0\}\cup Q_n(0)$, so $\pi=g\theta$ by Lemma \ref{unique}. Thus  $\pi h=g\theta h=g\pi$. Since $g$ is surjective, it must be injective, and so $g\in \Aut(Q_n)$. If $k\in K^\pi$, then $g^{-1}kg\pi=g^{-1}k\pi
 h=g^{-1}\pi h=\pi$, and so $g^{-1}kg\in K^\pi$. Thus $g\in N^\pi$. Recall that $N^\pi$ acts on
 $V\Pi$ by $(x\pi)^g:=(x^g)\pi$ for all $x\in VQ_n$. Hence $(x\pi)^ h=x(\pi h)=x(g\pi)=(x\pi)^g$ for all $x\in
 VQ_n$, and it follows that $N^\pi/K^\pi\simeq \Aut(\Pi)$.
\end{proof}

Next we see that $\Aut(\Pi)_u$ and $N_0^\pi$ are closely related. Note that if $(x,\sigma)\in
\Aut(Q_n)$ fixes $0$, then $x=0$. Hence we may view $N^\pi_0$ as a subgroup of $S_n$. In particular,
$N^\pi_0$ acts faithfully on $\tbinom{n}{i}$ for $i=1,2$.

\begin{lemma}
\label{stab} Let $\Pi$ be a rectagraph and $\pi:Q_n\to\Pi$ a covering. Let $0\pi:=u$. Then the
actions of $\Aut(\Pi)_u$ on $\Pi_i(u)$ and $N^\pi_0$ on $\tbinom{n}{i}$ are permutation isomorphic
for $i=1,2$.
\end{lemma}

\begin{proof}
As in the proof of Proposition \ref{K}, there is no loss of generality in assuming that
$\Pi=(Q_n)_{K^\pi}$ and $\pi$ is the natural map. By Lemma \ref{K cover} and Proposition \ref{K},
there is a group isomorphism $\psi:N_0^\pi\to \Aut(\Pi)_u$  defined by $(x\pi)^{\sigma\psi}:=(x^\sigma)\pi$ for all $x\in VQ_n$ and $\sigma\in N^\pi_0$. Since $\pi$ is a covering, there is a bijection $\varphi_1:[n]\to \Pi(u)$ defined by $i\mapsto
e_i\pi$ for all $i\in [n]$, and since $\Pi$ is a rectagraph of valency $n$, there is a bijection
$\varphi_2:\tbinom{n}{2}\to \Pi_2(u)$ defined by $\{i,j\}\mapsto e_{i,j}\pi$ for all $\{i,j\}\in
\tbinom{n}{2}$. It is routine to verify that $\psi$ is a permutation isomorphism with respect to
each of these bijections.
\end{proof}

Note that since $N^\pi_0\leq S_n$, the vector space $\mathbb{F}_2^n$ is naturally an
$\mathbb{F}_2N^\pi_0$-module; indeed, it is the permutation module of $N^\pi_0$ over
$\mathbb{F}_2$. Then the set $E_n$ of vectors in $\mathbb{F}_2^n$ with even weight is an
$\mathbb{F}_2N^\pi_0$-submodule of $\mathbb{F}_2^n$. Moreover, we have the following important
observation.

\begin{lemma}
\label{mod} Let $\Pi$ be a rectagraph and $\pi:Q_n\to\Pi$ a covering. Then $K^\pi\cap
\mathbb{F}_2^n$ is an $\mathbb{F}_2N^\pi_0$-submodule of the permutation module $\mathbb{F}_2^n$ that does not contain
$E_n$.
\end{lemma}

\begin{proof}
 If $x\in K^\pi\cap \mathbb{F}_2^n$ and
 $\sigma\in N^\pi_0$, then $(x^\sigma,1)=(0,\sigma)^{-1}(x,1)(0,\sigma)\in K^\pi$
 since $N^\pi_0$ normalises $K^\pi$, and  so $x^\sigma\in K^\pi\cap \mathbb{F}_2^n$. Thus $K^\pi\cap \mathbb{F}_2^n$ is an
 $\mathbb{F}_2N^\pi_0$-submodule of $\mathbb{F}_2^n$. If $K^\pi\cap \mathbb{F}_2^n\geq E_n$, then $e_{1,2}\in K^\pi\cap \mathbb{F}_2^n$,
 and so $0\pi=0^{e_{1,2}}\pi=e_{1,2}\pi$, but this is impossible since $\pi$ is a covering and the vertices $0$ and
 $e_{1,2}$ are both neighbours of $e_1$ in $Q_n$.
 \end{proof}

Here is one useful application of Lemma \ref{mod}, for which we require the following definition. Let $G\leq S_n$. Note that $E_n$ and $\{0,1\}$ are both $\mathbb{F}_2G$-submodules of the permutation module $\mathbb{F}_2^n$. Then we may define the \textit{heart of $G$ over $\mathbb{F}_2$}  to be the $\mathbb{F}_2G$-module $E_n/(E_n\cap \{0,1\})$.

\begin{lemma}
\label{heart irred} Let $\Pi$ be a rectagraph and $\pi:Q_n\to\Pi$ a covering. Suppose that
$K^\pi\leq \mathbb{F}_2^n$ and $N^\pi_0$ is transitive on $[n]$. If the heart of $N_0^\pi$ over $\mathbb{F}_2$  is irreducible, then $\Pi\simeq Q_n$ or $\Box_n$.
\end{lemma}

\begin{proof}
  Lemma \ref{mod} implies that $K^\pi$ is an $\mathbb{F}_2N^\pi_0$-submodule of $\mathbb{F}_2^n$
  that does not contain $E_n$. Since $N^\pi_0$ is transitive on $[n]$ and the heart of $N^\pi_0$ over
  $\mathbb{F}_2$   is irreducible, the only $\mathbb{F}_2N^\pi_0$-submodules of
  $\mathbb{F}_2^n$ are $\{0\}$, $\{0,1\}$, $E_n$ and $\mathbb{F}_2^n$ by \cite[Lemma 2]{Mor1980}.
  Thus $K^\pi=\{0\}$ or $\{0,1\}$. Since $\Pi\simeq (Q_n)_{K^\pi}$ by Proposition \ref{K}, it
  follows that $\Pi\simeq Q_n$ or $\Box_n$.
\end{proof}

Here are some sufficient conditions for $K^\pi$ to be contained in $\mathbb{F}_2^n$.

\begin{lemma}
\label{rho} Let $\Pi$ be a rectagraph and $\pi:Q_n\to\Pi$ a covering. Then the following hold.
\begin{itemize}
 \item[(i)] If $N^\pi_0$ is $S_n$ or $A_n$ where $n\geq 5$, then $K^\pi\leq \mathbb{F}_2^n$.
 \item[(ii)] If $N^\pi_0$ is $2$-transitive on $[n]$ and $n$ is not a power of 2, then $K^\pi\leq
   \mathbb{F}_2^n$.
\end{itemize}
\end{lemma}

\begin{proof}
Let $\rho:\mathbb{F}_2^n\rtimes S_n\to S_n$ be the natural projection map. Note that $K^\pi\rho$ is
a $2$-group, for $K^\pi\rho$ is a quotient of $K^\pi$, and the preimages of $\pi$ partition
$VQ_n=\mathbb{F}_2^n$ and have size $|K^\pi|$ by Lemma \ref{preimage K}.

Since $N^\pi_0$ normalises $K^\pi$ in $\Aut(Q_n)$, it follows that $N^\pi_0$ normalises $K^\pi\rho$
in $S_n$. In particular, if $N^\pi_0$ is $S_n$ or $A_n$ where $n\geq 5$, then since $K^\pi\rho$ is a
2-group, we must have $K^\pi\rho=1$. Thus $K^\pi\leq \mathbb{F}_2^n$, proving (i).

Now suppose that $N^\pi_0$ is $2$-transitive on $[n]$ and that there exists $\sigma\in K^\pi\rho$
where $i^\sigma=j$ and $i\neq j$. If $k\in[n]\setminus\{i\}$, then there exists $\tau\in N^\pi_0$
with $i^\tau=i$ and $j^\tau=k$, so $\tau^{-1}\sigma\tau$ maps $i$ to $k$ and  lies in $K^\pi\rho$.
Thus $K^\pi\rho$ is a transitive subgroup of $S_n$. In particular, $n$ divides $|K^\pi\rho|$, which
is a power of 2, proving (ii).
\end{proof}

Next we see that Proposition \ref{K} provides a uniform description of the automorphism groups of
many coset graphs of binary linear codes, including the graphs of Theorem \ref{main rect}.

\begin{lemma}
\label{code aut} Let $C$ be a binary linear $[n,r,d]$-code where $d\geq 5$. Then
$\Aut(\Gamma(C))=(\mathbb{F}_2^n/C)\rtimes \Aut(C)$.
\end{lemma}

\begin{proof} 
By Lemma \ref{code}, the natural map $\pi:Q_n\to \Gamma(C)$ is a covering and $\Gamma(C)$ is a
rectagraph. Then $K^\pi=C$ by Lemma \ref{K cover}, so $N^\pi=\mathbb{F}_2^n\rtimes \Aut(C)$. Thus
$\Aut(\Gamma(C))\simeq N^\pi/K^\pi\simeq (\mathbb{F}_2^n/C)\rtimes \Aut(C)$ by Proposition \ref{K}.
\end{proof}

We are now in a position to prove Theorem \ref{main rect}.

\begin{proof}[Proof of Theorem \ref{main rect}]
Suppose that $\Pi$ is a rectagraph with $a_2=0$ and $c_3=3$, and suppose that there exists $u\in V\Pi$ such that $|\Pi(u)|\geq 4$ and $\Aut(\Pi)_u$ acts
$4$-homogeneously on $\Pi(u)$. By \cite[Proposition 1.1.2]{BroCohNeu1989}, $\Pi$ is regular of valency $n$ for some $n\geq 4$, and by Lemma \ref{pi}, there exists a covering
$\pi:Q_n\to \Pi $ such that $0\pi=u$.  Then $\Pi\simeq
(Q_n)_{K^\pi}$ by Proposition \ref{K} and $N^\pi_0$ is a $4$-homogeneous subgroup of $S_n$ by Lemma
\ref{stab}.

Since $\Pi$ is a rectagraph with $a_2=0$ and $c_3=3$, we obtain $|\Pi_2(u)|=\tbinom{n}{2}$ and
$|\Pi_3(u)|=\tbinom{n}{3}$ by a counting argument. Thus $ 1+n+\tbinom{n}{2}+\tbinom{n}{3}\leq
|V\Pi|$. But $|V\Pi|$ divides $2^n$ by Lemma \ref{cover} (ii), so if $n\leq 6$, then $|V\Pi|=2^n$, in which case $K^\pi=\{0\}$
and $\Pi\simeq Q_n$.

Similarly, if $n=7$, then $|V\Pi|=2^{6}$ or $2^7$, and $\Pi\simeq Q_7$ in the latter case, so we may
assume that $|V\Pi|=2^{6}$. Then $|K^\pi|=2$. Since $N^\pi_0$ is $4$-homogeneous, it is also
$3$-homogeneous, and so it is $2$-transitive by \cite[Theorem 2]{LivWag1965}. Then $K^\pi\leq
\mathbb{F}_2^7$ by Lemma \ref{rho} (ii), so $K^\pi=\{0,x\}$ for some $0\neq x\in \mathbb{F}_2^7$. Then $x^\sigma=x$ for
all $\sigma\in N^\pi_0$ by Lemma \ref{mod}. Since $N^\pi_0$ is transitive, it follows that $x=1$,
and so $\Pi\simeq \Box_7$.

Thus we may assume that $n\geq 8$. Then $N^\pi_0$ is $3$-transitive by \cite[Theorem 2]{LivWag1965}.
First suppose that $N^\pi_0$ is not $4$-transitive. Then $(N^\pi_0,n)$ is one of $(\PSL_2(8),9)$,
$(\PGaL_2(8),9)$ or $(\PGaL_2(32),33)$ by \cite{Kan1972}, in which case $K^\pi\leq \mathbb{F}_2^n$
by Lemma \ref{rho} (ii). Since the heart  of $N^\pi_0$ over $\mathbb{F}_2$ is irreducible by
\cite{Mor1980}, it follows from Lemma \ref{heart irred} that $\Pi\simeq Q_n$ or $\Box_n$.

Hence we may assume that $n\geq 8$ and $N^\pi_0$ is $4$-transitive. By the classification of the
finite simple groups, it follows that $(N^\pi_0,n)$ is one of $(S_n,n)$ or $(A_n,n)$ for $n\geq 8$, or $(M_{n},n)$
for $n=11$, $12$, $23$ or $24$ (cf. \cite[Theorem 4.11]{Cam1999}). Again, we have
$K^\pi\leq \mathbb{F}_2^n$ by Lemma \ref{rho}.

If $\Pi\simeq Q_n$ or $\Box_n$, then we are done, so we may assume that $K^\pi\neq \{0\},\{0,1\}$.
By Lemma \ref{heart irred}, the heart  of $N^\pi_0$ over $\mathbb{F}_2$  is therefore reducible, so $N^\pi_0$ is $M_{23}$ or $M_{24}$ by \cite{Mor1980}. Note that $K^\pi$ is
an $\mathbb{F}_2N^\pi_0$-submodule of $\mathbb{F}_2^n$ distinct from $E_n$ or $\mathbb{F}_2^n$ by
Lemma \ref{mod}. Using {\sf MAGMA} \cite{Magma}, it can be checked that if $N^\pi_0=M_{23}$, then
$K^\pi$ is either the binary Golay code $C_{23}$, in which case $\Pi\simeq \Gamma(C_{23})$, or
$K^\pi$ is $C_{23}\cap E_{23}$, in which case $\Pi\simeq \Gamma(C_{23}).2$ by Lemma \ref{code} (iii). Similarly, if
$N^\pi_0=M_{24}$, then $K^\pi$ is the extended binary Golay code $C_{24}$, in which case $\Pi\simeq
\Gamma(C_{24})$.

Conversely, suppose that $\Pi$ is one of the graphs described in (i)-(v) of the statement of the
theorem. Then $\Pi= \Gamma(C)$ where $C$ is a binary linear $[n,r,d]$-code with $d\geq 7$ (this holds for (iii) by Lemma \ref{code} (iii)), so $\Pi$
is a rectagraph with $a_2=0$ and $c_3=3$ by Lemma \ref{code}. Moreover, by Lemma \ref{code aut},
$\Aut(\Pi)_{0+C}$ is $S_n$ in cases (i)-(ii), $M_{23}$ in cases (iii)-(iv), and $M_{24}$ in case
(v), and so $\Aut(\Pi)_{0+C}$ is $4$-homogeneous on $\Pi(0+C)$.
\end{proof}

Though Corollary \ref{Cameron} follows immediately from Theorem \ref{main rect} since the $n$-cube is the only rectagraph of valency $n$ for $n\leq 3$, we provide a direct
proof of this result that does not use the classification of the finite simple groups.

\begin{proof}[Proof of Corollary \ref{Cameron}]
Let $\Pi$ be a rectagraph with $a_2=0$ and $c_3=3$, where for some $u\in V\Gamma$, the action of $\Aut(\Pi)_u$ on $\Pi(u)$ is permutation isomorphic to the natural action of $S_n$ or $A_n$ on $[n]$. By \cite[Proposition 1.1.2]{BroCohNeu1989}, $\Pi$ is regular of valency $n$, and by Lemma \ref{pi}, there is a covering $\pi:Q_n\to \Pi $ such that
$0\pi=u$. Then $\Pi\simeq (Q_n)_{K^\pi}$ by Proposition \ref{K} and $N^\pi_0$ is $S_n$ or $A_n$ by
Lemma \ref{stab}. For $n\leq 3$, we have $\Pi\simeq Q_n$, and for $n=4$, we have $\Pi\simeq Q_4$ since $a_2=0$ and $c_3=3$,
so we may assume that $n\geq 5$. Then
$K^\pi\leq \mathbb{F}_2^n$ by Lemma \ref{rho} (i), and it is routine to verify that the heart  of $N_0^\pi$ over
$\mathbb{F}_2$ is irreducible. Thus $\Pi\simeq Q_n$ or $\Box_n$ by Lemma
\ref{heart irred}, and $n\geq 7$ when $\Pi\simeq \Box_n$ since $a_2=0$ and $c_3=3$. The converse is straightforward.
\end{proof}

Next we see that Corollary \ref{Brouwer} is a natural consequence of Theorem \ref{main rect}.

\begin{proof}[Proof of Corollary \ref{Brouwer}]
Note that $\Pi$ is  a rectagraph. Suppose that there exists $u\in V\Pi$ such that   $|\Pi(u)|\geq 5$ and $\Aut(\Pi)_u$ is $5$-transitive on $\Pi(u)$. Any $5$-transitive group is
$4$-homogeneous, so $\Pi$ is one of the bipartite graphs listed in Theorem \ref{main rect}. Since
$M_{23}$ is not 5-transitive on $23$ points, it follows  that $\Pi$ is $Q_n$ where $n\geq 5$, or $\Box_n$ where $n$
is even (by Lemma \ref{code}) and $n\geq 8$, or $\Gamma(C_{24})$. The converse is straightforward.
\end{proof}

To finish this section, we see that the automorphism group of a rectagraph $\Pi$ is closely related
to that of a connected component in the distance 2 graph $\Pi_2$ of $\Pi$.

\begin{lemma}
\label{aut2} Let $\Pi$ be a rectagraph of valency $n\geq 3$, and let $\Gamma$ be a connected
component of $\Pi_2$. Then the map $\varphi:\Aut(\Pi)_{V\Gamma}\to \Aut(\Gamma)$ defined by
$g\mapsto g|_{V\Gamma}$ for all $g\in \Aut(\Pi)_{V\Pi}$ is an injective group homomorphism.
\end{lemma}

\begin{proof}
 The map $\varphi$ is well-defined since automorphisms preserve distance, so $\varphi$ is a group
 homomorphism. If $\Pi$ is not bipartite, then $V\Gamma=V\Pi$, in which case $\varphi$ is injective,
 so we may assume that $\Pi$ is bipartite. Let $X:=V\Gamma$ and $Y:=V\Pi\setminus X$, and suppose that $g\in\Aut(\Pi)$ fixes $X$ pointwise. Let
$y_1\in Y$. Choose $y_2\in \Pi_2(y_1)$, and let $x_1$ and $x_2$ be the two vertices of $X$ lying in
the quadrangle determined by $y_1$ and $y_2$. Then $g$ either fixes or interchanges $y_1$ and $y_2$.
Let $y_3\in Y$ be adjacent to $x_1$ but distinct from $y_1$ and $y_2$, and let $x_3$ be the vertex
in $X$ distinct from $x_1$ lying in the unique quadrangle determined by $y_1$ and $y_3$. Again, $g$
either fixes or interchanges $y_1$ and $y_3$. Hence $g$ fixes $y_1$. As $y_1$ was arbitrary, it
follows that $g$ fixes $Y$ pointwise, and so $g=1$, as desired.
\end{proof}

We remark that $\Aut(\Pi)_{V\Gamma}$ and $ \Aut(\Gamma)$ need not be isomorphic in general: if $\Pi$ is
the $4$-cube, then $\Aut(\Pi)_{V\Gamma}=2^3\rtimes S_4$ while $\Aut(\Gamma)=2^4\rtimes S_4$.

%%%%%%%%%%%%%%%%%%%%%%%%%%%%%%%%%%%%%%%%%%%%%%%%%%%
%
%	Locally triangular graphs
%
%%%%%%%%%%%%%%%%%%%%%%%%%%%%%%%%%%%%%%%%%%%%%%%%%%%

\section{Locally triangular graphs} \label{loc tri}

Recall that for $n\geq 2$, the \textit{triangular graph} $T_n$ is the graph whose vertices are the
$2$-subsets of $\{1,\ldots,n\}$, where a pair of 2-subsets are adjacent whenever they intersect at
exactly one point. It is well-known that $\Aut(T_n)=S_n$ for $n\geq 3$ and $n\neq 4$, while
$\Aut(T_4)=S_4\times C_2$ (cf. \cite[Proposition 9.1.2]{BroCohNeu1989}).

Furthermore, recall that a graph $\Gamma$ is \textit{locally triangular} if $[\Gamma(u)]$ is
isomorphic to a triangular graph for all $u\in V\Gamma$ and
\textit{locally} $T_n$ if $[\Gamma(u)]\simeq T_n$ for all $u\in V\Gamma$. By \cite[Proposition
4.3.9]{BroCohNeu1989}, every connected locally triangular graph is locally $T_n$ for some $n$. Thus
there is no loss of generality in focusing on graphs that are locally $T_n$. Note that if $n$ is 2,
3 or 4, then $T_n$ is $K_1$, $K_3$ or $K_{3[2]}$ respectively, and so the only connected locally
$T_n$ graph is $K_2$, $K_4$ or $K_{4[2]}$ respectively.

Let $\Gamma$ be a graph that is locally $T_n$. If $n\geq 4$, then $\Gamma$ is not complete and has
two families of maximal cliques, corresponding to the two families of maximal cliques in $T_n$. A
maximal clique in $T_n$ consists either of the $2$-subsets of $[n]$ containing some fixed $i\in
[n]$, or the $2$-subsets of $[n]$ contained in some $3$-subset of $[n]$. In particular, the maximal
cliques in $\Gamma$ either have size 4 or $n$, and so they are easily distinguished for $n\geq 5$.

By \cite[Proposition 4.3.9]{BroCohNeu1989}, every connected graph $\Gamma$ that is locally $T_n$ is
a halved graph of some bipartite rectagraph $\Pi$ where $c_3(\Pi)=3$. If $n\leq 4$, then we may take
$\Pi$ to be the $n$-cube, and if $n\geq 5$, then $\Pi$ is defined to be the bipartite graph with
parts $V\Gamma$ and $\mathcal{B}$, where $\mathcal{B}$ is the set of $n$-cliques of $\Gamma$, and
vertices $u\in V\Gamma$ and $B\in \mathcal{B}$ are adjacent whenever $u\in B$. This was first
observed in \cite[Proposition 3]{Neu1985}. Moreover, we have the following.

\begin{lemma}
\label{unique rect} Let $\Pi$ be a bipartite rectagraph with halved graphs $\Gamma$ and $\Delta$. If
$\Gamma$ is locally $T_n$ where $n\geq 5$ and $\mathcal{B}$ is the set of $n$-cliques of $\Gamma$,
then there exists a bijection $\varphi:V\Delta\to\mathcal{B}$ such that $u\in V\Gamma$ is adjacent
to $v\in V\Delta$ if and only if $u\in v\varphi$. Furthermore, we have $c_3(\Pi)=3$.
\end{lemma}

\begin{proof}
Recall that $\Pi$ is regular by \cite[Proposition 1.1.2]{BroCohNeu1989}.
Since $|\Pi_2(u)|=|\Gamma(u)|=\tbinom{n}{2}$ for any $u\in V\Gamma$, a counting argument shows that
$\Pi$ has valency $n$. In particular, if $v\in V\Delta$, then $\Pi(v)$ is an $n$-clique in $\Gamma$,
and so $\Pi(v)\in \mathcal{B}$. Define $\varphi:V\Delta\to \mathcal{B}$ by $v\mapsto \Pi(v)$ for all
$v\in V\Delta$. Then $u\in V\Gamma$ is adjacent to $v\in V\Delta$ if and only if $u\in v\varphi$.
Moreover, $\varphi$ is injective since $\Pi$ is a rectagraph and $n\geq 3$. Since the parts of a
regular bipartite graph must have the same size, and since $V\Gamma$ and $\mathcal{B}$ are the parts
of a regular bipartite graph by the proof of \cite[Proposition 4.3.9]{BroCohNeu1989}, it follows
that $|V\Delta|=|V\Gamma|=|\mathcal{B}|$. Thus $\varphi$ is surjective. Since $\Pi$ is then
isomorphic to the bipartite rectagraph constructed in the proof of \cite[Proposition
4.3.9]{BroCohNeu1989}, we have $c_3(\Pi)=3$.
\end{proof}

Thus a connected graph that is locally $T_n$ is a halved graph of a unique bipartite rectagraph
(this is also the case for $n\leq 4$). Note, however, that such a graph may also be the distance 2
graph of a non-bipartite rectagraph.

Lemma \ref{unique rect} has various applications. For example, it follows from Lemma \ref{unique
rect} and a straightforward exercise that if $\Pi$ is a bipartite rectagraph and $\Gamma$ is a halved graph
of $\Pi$, then $\Gamma$ is locally triangular if and only if $c_3(\Pi)=3$. In addition, Lemma
\ref{unique rect} enables us to determine the relationship between the automorphism group of a
bipartite rectagraph and its locally triangular halved graph.

\begin{lemma}
\label{autgamma} Let $\Pi$ be a bipartite rectagraph, and let $\Gamma$ be a halved graph of $\Pi$
that is locally $T_n$ where $n\geq 5$. Then the following hold.
\begin{itemize}
\item[(i)] $ \Aut(\Gamma)\simeq \Aut(\Pi)_{V\Gamma}$.
\item[(ii)] The actions of $\Aut(\Gamma)_u$ and $\Aut(\Pi)_u$ on $\Gamma(u)=\Pi_2(u)$ are faithful
  and permutation isomorphic for all $u\in V\Gamma$.
\end{itemize}
\end{lemma}

\begin{proof}
 By Lemma \ref{unique rect}, we may identify $V\Pi\setminus V\Gamma$ with the set $\mathcal{B}$ of
 $n$-cliques of $\Gamma$ in such a way that $u\in V\Gamma$ is adjacent to $B\in \mathcal{B}$
 whenever $u\in B$. Note that $\Pi$ has valency $n\geq 5$. Then the map
 $\varphi:\Aut(\Pi)_{V\Gamma}\to \Aut(\Gamma)$ defined by $g\mapsto g|_{V\Gamma}$ for all $g\in
 \Aut(\Pi)_{V\Pi}$ is an injective group homomorphism by Lemma \ref{aut2}. Since $\Aut(\Gamma)$ acts
 naturally on $\mathcal{B}$, it follows that $\Aut(\Gamma)$ acts on $V\Pi=V\Gamma\cup \mathcal{B}$.
 This action is faithful and preserves adjacency in $\Pi$, and so $\varphi$ is surjective. Thus (i)
 holds.

Let $u\in V\Gamma$. Since $c_3(\Pi)=3$ by Lemma \ref{unique rect}, there exists a covering
$\pi:Q_n\to \Pi$ such that $0\pi=u$ by Lemma \ref{pi}. Then $\Aut(\Pi)_u$ acts faithfully on
$\Pi_2(u)$ by Lemma \ref{stab}. Observe that $\Aut(\Pi)_u\leq \Aut(\Pi)_{V\Gamma}$, for if $g\in
\Aut(\Pi)_u$ and $v\in V\Gamma$, then $d_\Pi(u,v^g)=d_\Pi(u,v)$, so $v^g\in V\Gamma$. Hence $
\Aut(\Pi)_u\simeq\{g|_{V\Gamma}:g\in \Aut(\Pi)_u\}= \Aut(\Gamma)_u$. Then $\Aut(\Gamma)_u$ acts
faithfully on $\Gamma(u)$, and the actions of $\Aut(\Gamma)_u$ and $\Aut(\Pi)_u$ on
$\Gamma(u)=\Pi_2(u)$ are permutation isomorphic.
\end{proof}

Note that Lemma \ref{autgamma} does not hold when $n=4$, for then $\Pi\simeq Q_4$, in which case
$\Aut(\Pi)_{V\Gamma}\not\simeq \Aut(\Gamma)$, as observed after the proof of Lemma \ref{aut2}, and
$\Aut(\Pi)_{u}\not\simeq \Aut(\Gamma)_u$, for $\Aut(\Pi)_{u}= S_4$ while $\Aut(\Gamma)_u= S_4\times
C_2$.

Lemma \ref{autgamma} has the following interesting consequence.

\begin{prop}
 \label{faithful} Let $\Gamma$ be a connected locally triangular graph. Then $\Aut(\Gamma)_u$ acts
 faithfully on $\Gamma(u)$ for all $u\in V\Gamma$.
\end{prop}

\begin{proof}
 By \cite[Proposition 4.3.9]{BroCohNeu1989}, there exists an integer $n\geq 2$ such that $\Gamma$ is
 locally $T_n$, and $\Gamma$ is a halved graph of some bipartite rectagraph $\Pi$. If $n\leq 4$,
 then the result is trivial since $\Gamma$ is $K_2$, $K_4$ or $K_{4[2]}$, and if $n\geq 5$, then the
 result follows from Lemma \ref{autgamma}.
\end{proof}

To finish this section, we see that there are many examples of graphs that are locally $T_n$.

\begin{lemma}
\label{code tri} Let $C$ be a binary linear $[n,r,d]$-code where $n\geq 2$ and $d\geq 7$. Then any
connected component of $\Gamma(C)_2$ is locally $T_n$.
\end{lemma}

\begin{proof}
Let $\Gamma$ be a connected component of $\Gamma(C)_2$, and let $x+C\in V\Gamma$. Define a map
$\varphi: T_n\to [\Gamma(x+C)]$ by $\{i,j\}\mapsto x+e_{i,j}+C$ for all $\{i,j\}\in \tbinom{n}{2}$.
It is routine to verify that $\varphi$ is a graph isomorphism.
\end{proof}

%%%%%%%%%%%%%%%%%%%%%%%%%%%%%%%%%%%%%%%%%%%%%%%%%%%
%
%	5. Locally rank 3 graphs
%
%%%%%%%%%%%%%%%%%%%%%%%%%%%%%%%%%%%%%%%%%%%%%%%%%%%

\section{Locally rank 3 graphs} \label{loc rank 3}

In this section, we prove Theorem \ref{main} and Corollary \ref{Hall Shult}. We begin with some
preliminary observations about locally rank 3 graphs.

Let $G$ be a transitive permutation group on $\Omega$, and let $\omega\in\Omega$. A
\textit{suborbit} of $G$ is an orbit of the stabiliser $G_\omega$ on $\Omega$, and the \textit{rank}
of $G$ is the number of suborbits of $G$ (which is independent of the choice of $\omega$ by
transitivity). The rank is also the number of orbits of $G$ on $\Omega\times\Omega$.

Recall that a graph $\Gamma$ is \textit{locally rank 3 with respect to $G$} if $\Gamma$ has no
vertices with valency $0$ and $G\leq \Aut(\Gamma)$ such that $G_{u}^{\Gamma(u)}$ is transitive of
rank 3 on $\Gamma(u)$ for every $u\in V\Gamma$. We also say that $\Gamma$ is \textit{locally rank 3}
when $\Gamma$ is locally rank 3 with respect to some $G$. Since $G_u$ acts transitively on
$\Gamma(u)$ for all $u\in V\Gamma$, we obtain the following.

\begin{lemma}
\label{ver tran} Let $\Gamma$ be a connected graph that is locally rank $3$ with respect to $G$.
Then $\Gamma$ is $G$-edge-transitive. If $\Gamma$ is not bipartite, then $\Gamma$ is also
$G$-vertex-transitive.
\end{lemma}

\begin{proof}
 Observe that any pair of vertices with a common neighbour lie in the same $G$-orbit. Then
 connectivity implies that $\Gamma$ is $G$-edge-transitive. If $\Gamma$ is not bipartite, then $\Gamma$
 contains an odd length cycle, so there exists an edge whose ends are in the same $G$-orbit. There
 is a path of even length from every vertex in $\Gamma$ to one of these ends, so $\Gamma$ is
 $G$-vertex transitive.
\end{proof}

If a graph $\Gamma$ is complete or has girth at least 4, then the induced neighbourhood graphs are
complete or have no edges, in which case the suborbits of $G_u^{\Gamma(u)}$ will depend entirely on
the automorphism group $G$. However, when $\Gamma$ is a locally rank 3 non-complete graph with girth
3, these suborbits can be described combinatorially.

\begin{lemma}
\label{orbits} Let $\Gamma$ be a connected non-complete graph with girth $3$ that is locally rank
$3$ with respect to $G$. Then for every $u\in V\Gamma$ and $v\in \Gamma(u)$, the orbits of $G_{u,v}$
on $\Gamma(u)$ are $\{v\}$, $\Gamma(u)\cap\Gamma(v)$ and $\Gamma(u)\cap\Gamma_2(v)$.
\end{lemma}

\begin{proof}
Fix $u\in V\Gamma$ and $v\in \Gamma(u)$. Clearly the sets $\{v\}$, $\Gamma(u)\cap\Gamma(v)$ and
$\Gamma(u)\cap\Gamma_2(v)$ are $G_{u,v}$-invariant and partition $\Gamma(u)$, so it suffices to show
that the latter two are non-empty. Observe that since $\Gamma$ has girth 3, it is not bipartite, and
so it is $G$-vertex-transitive by Lemma \ref{ver tran}. In particular, the induced neighbourhood
graphs of $\Gamma$ are all isomorphic. If $\Gamma(u)\cap\Gamma(v)=\varnothing$, then $[\Gamma(u)]$
has no edges, and so $\Gamma$ has no triangles, a contradiction. Similarly, if
$\Gamma(u)\cap\Gamma_2(v)= \varnothing$, then $[\Gamma(u)]$ is complete, and so $\Gamma$ is
complete, a contradiction.
\end{proof}

Using this result, we see that it is usually sufficient to consider the full automorphism group in
order to classify a locally rank 3 graph with girth 3.

\begin{lemma}
\label{full} Let $\Gamma$ be a connected non-complete graph with girth $3$ that is locally rank $3$
with respect to $G$. If $G\leq H\leq \Aut(\Gamma)$, then $\Gamma$ is locally rank $3$ with respect
to $H$.
\end{lemma}

\begin{proof}
Let $u\in V\Gamma$ and $v\in\Gamma(u)$. By Lemma \ref{orbits}, the orbits of $G_{u,v}$ on
$\Gamma(u)$ are $\{v\}$, $\Gamma(u)\cap\Gamma(v)$ and $\Gamma(u)\cap\Gamma_2(v)$. These are
$H_{u,v}$-invariant sets, so they are the orbits of $H_{u,v}$ on $\Gamma(u)$.
\end{proof}

Note that Lemma \ref{full} does not hold in general. For example, the
complete graph $K_4$ is locally rank 3 with respect to $A_4$ but not $\Aut(K_4)=S_4$. Moreover, let
$n:=\tbinom{m}{2}$ where $m\geq 4$, and view $S_m$ as a subgroup of $S_n$ via the action of $S_m$ on
$\tbinom{m}{2}$. Then $Q_n$ has girth 4 and is locally rank 3 with respect to $\mathbb{F}_2^n\rtimes
S_m$ but not $\Aut(Q_n)=\mathbb{F}_2^n\rtimes S_n$. Thus $K_4$ and $Q_n$ (for $n=\tbinom{m}{2}$ and
$m\geq 4$) are examples of graphs that are locally rank 3 and locally $2$-arc transitive.

Recall that a $2$-arc $(u,v,w)$ is either a \textit{triangle} when $u$ and $w$ are adjacent, or a
\textit{$2$-geodesic} when $u$ and $w$ are distance 2 apart. For any $u\in V\Gamma$, the (possibly
empty) sets of triangles and $2$-geodesics with initial vertex $u$ are $G_u$-invariant. It turns out
that for a connected non-complete graph with girth 3, these two sets are $G_u$-orbits precisely when
$\Gamma$ is locally rank 3 with respect to $G$. This result, which we now prove, is somewhat
surprising, for the assumption that these two sets form orbits does not, at first glance, seem
likely to imply that $G_u$ is transitive on $\Gamma(u)$ for all $u\in V\Gamma$.

\begin{prop}
 \label{2 orbits} Let $\Gamma$ be a connected non-complete graph with girth $3$, and let $G\leq
 \Aut(\Gamma)$. Then the following are equivalent.
 \begin{itemize}
 \item[(i)] $\Gamma$ is locally rank $3$ with respect to $G$. 
 \item[(ii)] For each $u\in V\Gamma$, there are two orbits of $G_u$ on the $2$-arcs starting at $u$, namely the set of  triangles starting at $u$ and the set of $2$-geodesics starting at $u$.
 \end{itemize}
 \end{prop}

\begin{proof}
Suppose that for each $u\in V\Gamma$, there are two orbits of $G_u$ on the $2$-arcs starting at $u$,
namely the sets of triangles and $2$-geodesics. Let $u\in V\Gamma$. First we claim that $G_u$ acts
transitively on $\Gamma(u)$. Since $u$ lies in a triangle, there exists $v\in \Gamma(u)$ and $w\in
\Gamma(u)\cap\Gamma(v)$. Let $x\in \Gamma(u)$, and suppose for a contradiction that
$\Gamma(u)\cap\Gamma(x)=\varnothing$. Since $x$ lies in a triangle and therefore does not have
valency 1, there exists $y\in \Gamma_2(u)\cap\Gamma(x)$. Moreover, since $(w,u,v)$ and $(w,v,u)$
are triangles, there is some $g \in G_w$ that interchanges $u$ and $v$. 
In particular, since $v$ and $x$ are not adjacent, $x^g$ lies in $\Gamma_2(u)\cap \Gamma(v)$. 
Then $(u,v,x^g)$ and $(u,x,y)$ are $2$-geodesics, so $v^h=x$ for some $h\in G_u$, in which case $w^h\in \Gamma(u)\cap\Gamma(x)$, a
contradiction. Thus there exists $y\in \Gamma(u)\cap\Gamma(x)$, and since $(u,v,w)$ and $(u,x,y)$
are triangles, we have $x\in v^{G_u}$, as desired.

Now we determine the orbits of $G_{u,v}$ on $\Gamma(u)$, where $v\in \Gamma(u)$. If $x,y\in
\Gamma(u)\cap\Gamma(v)$, then $(v,u,x)$ and $(v,u,y)$ are triangles, so there exists $g\in G_{u,v}$
such that $x^g=y$. Similarly, if $x,y\in \Gamma(u)\cap\Gamma_2(v)$, then $(v,u,x)$ and $(v,u,y)$ are
$2$-geodesics, so there exists $g\in G_{u,v}$ such that $x^g=y$. Thus the orbits of $G_{u,v}$ on
$\Gamma(u)$ are $\{v\}$, $\Gamma(u)\cap\Gamma(v)$ and $\Gamma(u)\cap\Gamma_2(v)$, and so $\Gamma$ is
locally rank 3 with respect to $G$.

Conversely, suppose that $\Gamma$ is locally rank 3 with respect to $G$. Let $u\in V\Gamma$. Recall
that the sets of triangles and $2$-geodesics starting at $u$ are $G_u$-invariant. They are also
non-empty, for there exists $v\in \Gamma(u)$ (since $\Gamma$ has no vertices of valency 0 by assumption), and so there exists $w\in \Gamma(v)\cap \Gamma(u)$ and
$x\in \Gamma(v)\cap\Gamma_2(u)$ by Lemma \ref{orbits}, in which case $(u,v,w)$ is a triangle and
$(u,v,x)$ is a $2$-geodesic.

Now let $(u,v,w)$ and $(u,x,y)$ be 2-arcs. Since $G_u$ is transitive on $\Gamma(u)$, there exists
$g\in G_u$ such that $v^g=x$. If $(u,v,w)$ and $(u,x,y)$ are triangles, then $(u,x,w^g)$ is a
triangle, so $w^g,y\in \Gamma(x)\cap \Gamma(u)$, and by Lemma \ref{orbits} there exists $h\in
G_{x,u}$ such that $w^{gh}=y$. Thus $(u,v,w)^{gh}=(u,x,y)$, so the set of triangles starting at $u$ is
a $G_u$-orbit. Similarly, if $(u,v,w)$ and $(u,x,y)$ are $2$-geodesics, then $(u,x,w^g)$ is a $2$-geodesic, so
$w^g,y\in \Gamma(x)\cap\Gamma_2(u)$, and by Lemma \ref{orbits} there exists $h\in G_{x,u}$ such that
$w^{gh}=y$. Thus $(u,v,w)^{gh}=(u,x,y)$, so the set of $2$-geodesics starting at $u$ is a
$G_u$-orbit.
\end{proof}

Next we provide some  results  that will be used in the proof of Theorem \ref{main}. The first gives us information about the automorphism groups of
halved coset graphs of binary linear codes. Recall that if $C$ is a binary linear $[n,r,d]$-code,
then $\Aut(C)$ is the subgroup of $S_n$ that preserves~$C$.

\begin{lemma}
\label{rank 3 code} Let $C$ be an even binary linear $[n,r,d]$-code where $n\geq 5$ and $d\geq 7$.
Let $\Gamma$ be a halved graph of $\Gamma(C)$. Then $\Aut(\Gamma)=(E_n/C)\rtimes \Aut(C)$. Moreover,
let $\rho:\Aut(\Gamma)\to \Aut(C)$ be the natural projection map, and let $G\leq \Aut(\Gamma)$. Then
the following hold.
\begin{itemize}
 \item[(i)] The actions of $G_{x+C}$ on $\Gamma(x+C)$ and $(G_{x+C})\rho$ on $\tbinom{n}{2}$ are
   permutation isomorphic for all $x\in E_n$.
\item[(ii)] If $G=(E_n/C)\rtimes H $ where $H\leq \Aut(C)$ is transitive of rank $3$ on
  $\tbinom{n}{2}$, then $\Gamma$ is locally rank $3$ with respect to $G$.
\item[(iii)] If $H:=(G_{0+C})\rho=(G_{x+C})\rho$ for all $x\in E_n$ and $H$ is transitive on
  $\tbinom{n}{2}$, then $G=(E_n/C)\rtimes H$.
\end{itemize}
\end{lemma}

\begin{proof}
 By Lemma \ref{code}, we may assume that $V\Gamma=E_n/C$. By Lemma \ref{code tri} and its proof, the
 graph $\Gamma$ is locally $T_n$, and for each $x\in E_n$, there is a bijection
 $\varphi_x:\Gamma(x+C)\to \tbinom{n}{2}$ defined by $x+e_{i,j}+C\mapsto \{i,j\}$ for all $ i<j\leq
 n$. Recall from Lemma \ref{code} (v) that $\Gamma(C)$ is a rectagraph and recall from Lemma \ref{code aut} that $\Aut(\Gamma(C))=(\mathbb{F}_2^n/C)\rtimes \Aut(C)$. Then
 $\Aut(\Gamma)=(E_n/C)\rtimes \Aut(C)$ by Lemma \ref{autgamma}.

 Let $x\in E_n$. Since $G_{x+C}=G\cap \{(x+x^{\sigma}+C,\sigma^{-1}):\sigma\in \Aut(C)\}$, the
 actions of $G_{x+C}$ on $\Gamma(x+C)$ and $(G_{x+C})\rho$ on $\tbinom{n}{2}$ are permutation
 isomorphic under the group isomorphism $\rho|_{G_{x+C}}$ and the bijection $\varphi_x$. This proves
 (i).

 If $G=(E_n/C)\rtimes H $ where $H\leq \Aut(C)$ is transitive of rank 3 on $\tbinom{n}{2}$, then
 $\Gamma$ is $G$-vertex-transitive and $(G_{0+C})\rho=H$, so $\Gamma$ is locally rank 3 with respect
 to $G$ by (i). Thus (ii) holds.

Now suppose that $H:=(G_{0+C})\rho=(G_{x+C})\rho$ for all $x\in E_n$, and suppose that $H$ is transitive on
  $\tbinom{n}{2}$. Since
$G_{0+C}=\{(0+C,\sigma):\sigma\in H\}$, it follows that $(x+x^\sigma+C,1)\in G$ for all $x\in E_n$
and $\sigma\in H$. Let $i,j,k\in [n]$ be pairwise distinct. There exists $\sigma\in H$ such
that $\{i,k\}^\sigma=\{k,j\}$, so $(e_{i,j}+C,1)=(e_{i,k}+e_{i,k}^\sigma+C,1)\in G$. Hence $E_n/C\leq
G$, and so $G=(E_n/C)\rtimes H$, proving (iii).
\end{proof}

In the following, we invoke the classification of the finite simple groups to determine which
subgroups of $S_n$ are transitive of rank 3 on $\tbinom{n}{2}$.

\begin{thm}
 \label{rank 3} Let $H\leq S_n$ where $n\geq 5$. Then $H$ is transitive of rank $3$ on
 $\tbinom{n}{2}$ if and only if $(H,n)$ is one of $(S_n,n)$ or $(A_n,n)$ for $n\geq 5$,
 $(\PGaL_2(8),9)$, or $(M_{n},n)$ for $n=11$, $12$, $23$ or $24$.
\end{thm}

\begin{proof}
If $H$ is transitive of rank $3$ on $\tbinom{n}{2}$, then by \cite[Lemma 5]{Hig1970}, either $H$ is
$4$-transitive or $(H,n)$ is one of $(A_5,5)$ or $(\PGaL_2(8),9)$. By the classification of the
finite simple groups, the $4$-transitive subgroups of $S_n$ are $S_n$ for $n\geq 5$, $A_n$ for $n\geq 6$,  and  $M_{n}$ for $n=11$, $12$, $23$ or $24$ (cf. \cite[Theorem 4.11]{Cam1999}). Conversely, all of these groups are transitive of rank 3 on $\tbinom{n}{2}$.
\end{proof}

Surprisingly, using Theorem \ref{rank 3}, we can prove that the converse to Lemma \ref{rank 3 code}
(ii) holds.

\begin{prop}
\label{G} Let $C$ be an even binary linear $[n,r,d]$-code where $n\geq 5$ and $d\geq 7$. A halved
graph of $\Gamma(C)$ is locally rank $3$ with respect to $G$ if and only if $G=(E_n/C)\rtimes H$
where $H\leq \Aut(C)$ is transitive of rank $3$ on $\tbinom{n}{2}$.
\end{prop}

\begin{proof}
Let $\Gamma$ be a halved graph of $\Gamma(C)$, and suppose that $\Gamma$ is locally rank 3 with
respect to $G$. Then $\Aut(\Gamma)= (E_n/C)\rtimes \Aut(C)$ by Lemma \ref{rank 3 code}. Again, we may assume
that $V\Gamma=E_n/C$ by Lemma \ref{code}. Let $\rho:\Aut(\Gamma)\to \Aut(C)$ be the natural
projection map and $H:=(G_{0+C})\rho$. Then by Lemma \ref{rank 3 code} (i), $H$ is transitive of
rank 3 on $\tbinom{n}{2}$, and so it suffices to show that $H=(G_{x+C})\rho$ for all $x\in E_n$, for
then $G=(E_n/C)\rtimes H$ by Lemma \ref{rank 3 code} (iii), as desired.

Since $H$ is transitive of rank 3 on $\tbinom{n}{2}$ and $H\leq G\rho\leq S_n$ where $n\geq 5$, the
group $G\rho$ is also transitive of rank 3 on $\tbinom{n}{2}$. Moreover, $(G_{x+C})\rho$ is
transitive of rank 3 on $\tbinom{n}{2}$ for all $x\in E_n$ by Lemma \ref{rank 3 code} (i). Thus
$(G\rho,n)$ and $((G_{x+C})\rho,n)$ for all $x\in E_n$ are one of $(S_n,n)$ or $(A_n,n)$ for $n\geq 5$,
$(\PGaL_2(8),9)$, or $(M_{n},n)$ for $n=11$, $12$, $23$ or $24$ by Theorem
\ref{rank 3}.

If $(G\rho,n)$ is one of $(\PGaL_2(8),9)$ or $(M_{n},n)$ where $n=11$, $12$, $23$ or $24$, then
$(G_{x+C})\rho=G\rho$ for all $x\in E_n$, as desired, so we may assume that $A_n\leq G\rho$. Fix
$x\in E_n$. Since $|(G_{x+C})\rho|=|G_{x+C}|$ by Lemma \ref{rank 3 code} (i), the index
$[G\rho:(G_{x+C})\rho]$ divides $ [G:G_{x+C}]$, but $G$ acts transitively on $V\Gamma$ by Lemma
\ref{ver tran} since $\Gamma$ contains triangles, and $|V\Gamma|\mid 2^{n}$, so $[G\rho:(G_{x+C})\rho]=2^r$ for some $r\geq 0$. Since
$(G_{x+C})\rho$ is one of $S_n$, $A_n$, $\PGaL_2(8)$, or $M_{n}$ for $n=11$, $12$, $23$ or $24$, we
must have $A_n\leq (G_{x+C})\rho$. Since $G$ acts transitively on $V\Gamma$, the groups $G_{0+C}$
and $G_{x+C}$ are conjugate in $G$, and so the groups $H$ and $(G_{x+C})\rho$ are  conjugate in  $G\rho$. Thus
$H=(G_{x+C})\rho$.

The converse is Lemma \ref{rank 3 code} (ii), so the proof is complete.
\end{proof}

Recall from \S \ref{coset graphs} that if $C$ is a linear code in $\mathbb{F}_2^n$ that is not even, then $\Gamma(C)_2\simeq
\tfrac{1}{2}\Gamma(C\cap E_n)$. Thus Proposition \ref{G} can be used for codes that are not even as well.

Next we deal with small $n$. When $n=3$, the only locally $T_3\simeq K_3$ graph is $K_{4}$, which
has automorphism group $S_4$.

\begin{lemma}
\label{n=3} $K_{4}$ is locally rank $3$ with respect to $G$ if and only if $G=A_4$.
\end{lemma}

\begin{proof}
Suppose that $\Gamma:=K_4$ is locally rank 3 with respect to $G$. Clearly $G_u$ acts faithfully on
$\Gamma(u)$ for all $u\in V\Gamma$. Since $S_3$ has rank 2 on $\tbinom{3}{2}$, it follows that
$G_u\simeq A_3$ for all $u\in V\Gamma$, and so $G$ contains all of the $3$-cycles of $A_4$. Thus
$G=A_4$, and $\Gamma$ is indeed locally rank 3 with respect to $G$.
\end{proof}

Similarly, when $n=4$, the only locally $T_4\simeq K_{3[2]}$ graph is $K_{4[2]}$, which has
automorphism group $S_2\wr S_4$. In the following, we write elements of $S_2\wr S_4$ in the form $(\sigma_1,\sigma_2,\sigma_3,\sigma_4)\sigma$, where $(\sigma_1,\sigma_2,\sigma_3,\sigma_4)\in S_2^4$ and $\sigma\in S_4$. We also write $(12)$ for the transposition in $S_2$ and $S_4$ that interchanges 1
and 2.

\begin{lemma}
\label{n=4} The graph $K_{4[2]}$ is locally rank $3$ with respect to $G$ if and only if $G$ is one
of $S_2\wr S_4$, $E\rtimes S_4$ or $(E\rtimes A_4)\langle \tau\rangle$, where
$E:=\{(\sigma_1,\sigma_2,\sigma_3,\sigma_4)\in S_2^4: |\{i:\sigma_i\neq
1\}|\equiv 0\mod 2\}$ and $\tau:=((12),1,1,1)(12)$.
\end{lemma}

Note that if we identify $S_2$ with $\mathbb{F}_2$, then the subgroup $E$ of $S_2^4$ defined in
Lemma \ref{n=4} corresponds to the subgroup $E_4$ of $\mathbb{F}_2^4$.

\begin{proof}
Suppose that $\Gamma:=K_{4[2]}$ is locally rank 3 with respect to $G$. Then $\Aut(\Gamma)=S_2\wr
S_4$. Clearly $G_u$ acts faithfully on $\Gamma(u)$ for all $u\in V\Gamma$. If $G=S_2\wr S_4$, then we are done, so we may assume that $G<S_2\wr S_4$.  Fix $u\in V\Gamma$, and let $v,w\in \Gamma(u)$ correspond to $\{1,2\}$ and $\{1,3\}$ respectively.
Then $G_{u,v}$ has an orbit of size 4 containing $w$. Since $G$ is transitive on $V\Gamma$ by Lemma
\ref{ver tran}, and since $G_u$ is transitive on $\Gamma(u)$, we have that
$|G|=|G_u|8=|G_{u,v}|48=|G_{u,v,w}| 2^3\cdot 24$, and so $G_{u,v,w}=1$ and $G$ has index 2 in
$S_2\wr S_4$. Let $\rho:S_2^4\rtimes S_4\to S_4$ be the natural projection map. Then $G\leq
S_2^4\rtimes G\rho$, so $12\leq |G\rho|$. Note that $G\rho\neq A_4$, or else $G= S_2\wr A_4$ and
$G_{u,v,w}=S_2$, a contradiction. Hence $G\rho=S_4$, so $G$ is either $E\rtimes S_4$ or $(E\rtimes
A_4)\langle \tau\rangle$. Conversely, $\Gamma$ is indeed locally rank 3 with respect to $S_2\wr S_4$, $E\rtimes S_4$ and $(E\rtimes A_4)\langle \tau\rangle$.
\end{proof}

Now we prove Theorem \ref{main}.

\begin{proof}[Proof of   Theorem \ref{main}]
For this proof, recall that the bipartite double $\Gamma(C_{23}).2$ is isomorphic to $\Gamma(C_{23}\cap E_{23})$ by Lemma \ref{code} (iii).

Let $\Gamma$ be a connected graph that is locally triangular and locally rank 3 with respect to $G$.
We wish to show that $\Gamma$ is one of the graphs in (i)-(iv) and $G$ is one of the groups in
Table \ref{tab: group}.

The graph $\Gamma$ is locally $T_n$ for some integer $n\geq 2$ by \cite[Proposition
4.3.9]{BroCohNeu1989}. If $n=2$, then $G_u^{\Gamma(u)}$ has rank 1, a contradiction. If $n=3$, then
$\Gamma$ is locally $K_3$, and so $\Gamma\simeq K_4\simeq \tfrac{1}{2}Q_3$ and $G$ is one of the
groups in Table \ref{tab: group} by Lemma \ref{n=3}. If $n=4$, then $\Gamma$ is locally $K_{3[2]}$,
and so $\Gamma\simeq K_{4[2]}\simeq \tfrac{1}{2} Q_4$ and $G$ is one of the groups in Table
\ref{tab: group} by Lemma \ref{n=4}. Thus we may assume that $n\geq 5$.

By \cite[Propositions 1.1.2 and 4.3.9]{BroCohNeu1989}, $\Gamma$ is a halved graph of a bipartite rectagraph
$\Pi$ of valency $n$ with $c_3(\Pi)=3$. Let $u\in V\Gamma$. By Lemma \ref{pi}, there exists a
covering $\pi:Q_n\to \Pi$ such that $0\pi=u$. Since $n\geq 5$, the group $N^\pi_0$ is transitive of
rank 3 on $\tbinom{n}{2}$ by Lemmas \ref{stab}, \ref{autgamma} (ii) and \ref{full}. It is then routine to verify that  $N^\pi_0$ acts $4$-homogeneously on $[n]$, and so $\Aut(\Pi)_u$ acts $4$-homogeneously on
$\Pi(u)$ by Lemma \ref{stab}.

Since $\Pi$ is bipartite, it follows from  Theorem \ref{main rect} and Lemma \ref{code} (i)  that $\Pi$ is isomorphic to the
coset graph $\Gamma(C)$ where one of the following occurs: $C=\{0\}$ and $n\geq 5$; $C=\{0,1\}$, $n$
is even and $n\geq 8$; $C=C_{23}\cap E_{23}$ and $n=23$; or $C=C_{24}$ and $n=24$. Since the halved
graphs of $\Pi$ are then isomorphic by Lemma \ref{code} (ii), the graph $\Gamma$ is \textit{the}
halved graph of $\Pi$. Thus $\Gamma$ is one of the graphs described in (i)-(iv) where $n\geq 5$, and
so $G$ is one of the groups in Table \ref{tab: group} by Theorem \ref{rank 3} and Proposition
\ref{G}.

Conversely, let $\Gamma$ be one of the graphs in (i)-(iv) and $G$ one of the groups in Table
\ref{tab: group}. Then $\Gamma$ is locally triangular by Lemma \ref{code tri} . If $n=3$ or $4$, then $\Gamma$ is isomorphic to $K_4$ or $K_{4[2]}$ respectively, so $\Gamma$ is locally rank 3 with respect to $G$ by Lemmas \ref{n=3} and \ref{n=4}. If $n\geq 5$, then $\Gamma$ is locally rank 3 with respect to $G$ by Theorem \ref{rank 3} and Proposition \ref{G}.
\end{proof}

Lastly, we prove Corollary \ref{Hall Shult}.

\begin{proof}[Proof of Corollary \ref{Hall Shult}]
Suppose that there exists $G\leq \Aut(\Gamma)$ such that, for all  $u\in V\Gamma$, the action of $G_u$ on $\Gamma(u)$ is permutation isomorphic to the action of $H$  on $\tbinom{n}{2}$, where $H\leq S_n$ is transitive of rank 3 on
$\tbinom{n}{2}$ and $n\geq 5$. Let $u\in V\Gamma$ and $v\in \Gamma(u)$. Lemma \ref{orbits} implies that the orbits of $G_{u,v}$ on
$\Gamma(u)$ are $\{v\}$, $\Gamma(u)\cap\Gamma(v)$ and $\Gamma(u)\cap \Gamma_2(v)$. Without loss of
generality, we may assume that $v$ corresponds to $\{1,2\}$. Note that the orbits of $H_{\{1,2\}}$
on $\tbinom{n}{2}\setminus\{\{1,2\}\}$ consist of the set $X$ of $2$-subsets containing either 1 or 2, and the set $Y$
of $2$-subsets containing neither 1 nor 2. If $X=\Gamma(u)\cap\Gamma(v)$, then $[\Gamma(u)]\simeq T_n$, and if $X=\Gamma(u)\cap\Gamma_2(v)$, then $[\Gamma(u)]\simeq\overline{T}_n$. Since $\Gamma$ has girth 3, it is $G$-vertex transitive by Lemma \ref{ver tran}, and so $\Gamma$ is either
locally $T_n$ or locally $\overline{T}_n$. In the former case, Theorem \ref{main} applies, and in
the latter case, \cite[Theorem 2]{HalShu1985} applies.

Conversely, if $\Gamma$ is a graph from Theorem \ref{main}, then $\Gamma$ is locally rank 3 and
locally $T_n$, so the claim holds. If $\Gamma=\overline{T}_{n+2}$ where $n\geq 5$, then
$\Aut(\Gamma)_u=S_n\times C_2$ and $\Aut(\Gamma)_u^{\Gamma(u)}=S_n$ for all $u\in V\Gamma$, in which
case the claim holds with $G=\Aut(\Gamma)$. We obtain the same result when $n=5$ and $\Gamma$ is the
Conway-Smith graph by \cite[Theorem 13.2.3]{BroCohNeu1989}, or when $n=5$ and $\Gamma$ is the
commuting involutions graph of the conjugacy class of the involutory Galois field automorphism in
$\PSigmaL_2(25)$ by \cite[Proposition 12.2.2]{BroCohNeu1989}, or when $n=6$ and $\Gamma$ is the
complement of an elliptic quadric in the graph $\Sp_6(2)$ by {\sf GAP}
\cite{GAP4,FinInG,Grape}. Lastly, if $n=6$ and $\Gamma$ is the complement of a hyperplane in the
graph $\Sp_6(2)$, then 
$\Aut(\Gamma)_u^{\Gamma(u)}\simeq\Aut(\Gamma)_u=S_6$ for all $u\in V\Gamma$ by {\sf GAP} \cite{GAP4,FinInG,Grape},  in which case the claim
holds with $G=\Aut(\Gamma)$. \end{proof}

%%%%%%%%%%%%%%%%%%%%%%%%%%%%%%%%%%%%%%%%%%%%%%%%%%%
%
%	Acknowledgements
%
%%%%%%%%%%%%%%%%%%%%%%%%%%%%%%%%%%%%%%%%%%%%%%%%%%%

\section*{Acknowledgements} 

The first author acknowledges the support of the Australian Research Council Future Fellowship FT120100036.
The third author acknowledges the support of the Australian Research Council Discovery Grant DP35000000.

\bibliographystyle{acm}
\bibliography{jbf_references}

\end{document}